\documentclass[12pt,reqno]{amsart}
\usepackage{url}
\usepackage{graphics}
\usepackage{color}

\newtheorem{Theorem}{Theorem}[section]
\newtheorem{Lemma}[Theorem]{Lemma}
\newtheorem{Corollary}[Theorem]{Corollary}
\newtheorem{Proposition}[Theorem]{Proposition}

\newtheorem{Conjecture}[Theorem]{Conjecture}
\newtheorem{Question}[Theorem]{Question}

\newtheorem{Conj}{Conjecture}

\def\qed{\ifhmode\textqed\fi
	\ifmmode\ifinner\hfill\quad\qedsymbol\else\dispqed\fi\fi}
\def\textqed{\unskip\nobreak\penalty50
	\hskip2em\hbox{}\nobreak\hfill\qedsymbol
	\parfillskip=0pt \finalhyphendemerits=0}
\def\dispqed{\rlap{\qquad\qedsymbol}}

\def\p{\mathfrak{p}}

\def\v{\textup{v}}
\def\ZZ{\mathbb{Z}}

\def\HS{\textup{HS}}
\def\set{\textup{set}}
\def\pd{\textup{proj\,dim}}
\def\depth{\textup{depth}}
\def\reg{\textup{reg}}
\def\lcm{\textup{lcm}}
\def\supp{\textup{supp}}

\def\lex{\textup{lex}}
\def\Tor{\textup{Tor}}
\def\Ass{\textup{Ass}}
\def\Lin{\textup{Lin}}
\def\Deg{\textup{Deg}}
\def\Ker{\textup{Ker}}
\def\height{\textup{height}}
\def\Min{\textup{Min}}

\usepackage{lipsum}

\begin{document}
	
	\title{Edge ideals and their asymptotic syzygies}	
	\author{Antonino Ficarra, Ayesha Asloob Qureshi}
	
	\address{Antonino Ficarra, Departamento de Matem\'{a}tica, Escola de Ci\^{e}ncias e Tecnologia, Centro de Investiga\c{c}\~{a}o, Matem\'{a}tica e Aplica\c{c}\~{o}es, Instituto de Investiga\c{c}\~{a}o e Forma\c{c}\~{a}o Avan\c{c}ada, Universidade de \'{E}vora, Rua Rom\~{a}o Ramalho, 59, P--7000--671 \'{E}vora, Portugal}
	\email{antonino.ficarra@uevora.pt}\email{antficarra@unime.it}
	
	\address{Ayesha Asloob Qureshi, Sabanci University, Faculty of Engineering and Natural Sciences, Orta Mahalle, Tuzla 34956, Istanbul,
		Turkey}
	\email{ayesha.asloob@sabanciuniv.edu}\email{aqureshi@sabanciuniv.edu}
	
	\thanks{
	}
	
	
	
	\begin{abstract}
		Let $G$ be a finite simple graph, and let $I(G)$ denote its edge ideal. In this paper, we investigate the asymptotic behavior of the syzygies of powers of edge ideals through the lens of homological shift ideals $\text{HS}_i(I(G)^k)$. We introduce the notion of the $i$th homological strong persistence property for monomial ideals $I$, providing an algebraic characterization that ensures the chain of inclusions $\text{Ass}\,\text{HS}_i(I)\subseteq\text{Ass}\,\text{HS}_i(I^2)\subseteq\text{Ass}\,\text{HS}_i(I^3) \subseteq\cdots$. We prove that edge ideals possess both the 0th and 1st homological strong persistence properties. To this end, we explicitly describe the first homological shift algebra of $I(G)$ and show that $\text{HS}_1(I(G)^{k+1}) = I(G) \cdot \text{HS}_1(I(G)^k)$ for all $k \ge 1$. Finally, we conjecture that if $I(G)$ has a linear resolution, then $\text{HS}_i(I(G)^k)$ also has a linear resolution for all $k \gg 0$, and we present partial results supporting this conjecture.
	\end{abstract}
	
	\maketitle
	\vspace*{-0.8cm}
	\section*{Introduction}
	
	Let $S = K[x_1, \dots, x_n]$ be the standard graded polynomial ring over a field $K$, and let $I \subset S$ be a graded ideal. The asymptotic behavior of the powers of $I$ has been a central theme in Commutative Algebra for several decades, motivated by foundational questions posed by David Rees in the 1960s. In the 1970s, a number of significant contributions addressed the structure and properties of the set of associated primes $\Ass\,I^k$ of the powers of $I$. A major breakthrough came in 1979 when Brodmann \cite{B79} established that for any ideal $I$ in a Noetherian ring $R$, the sets of associated primes stabilize, that is, $\Ass\,I^k = \Ass\,I^{k+1}$ for all $k \gg 0$. This result, now a cornerstone in the field, led to the definition of the stable set of associated primes, denoted by $\Ass^\infty I$. Shortly thereafter, Brodmann also demonstrated that the depth of the powers $\depth\,R/I^k$ eventually stabilizes \cite{B79a}.
	
	Following these developments, the asymptotic linearity of other invariants of powers of ideals has attracted extensive interest. In 1999, Kodiyalam \cite{Kod99} and, independently, Cutkosky, Herzog, and Trung \cite{CHT99} proved that the Castelnuovo-Mumford regularity $\reg\,I^k$ behaves as a linear function for all $k\gg0$. More recently, Conca \cite{Conca23} and, independently, Ficarra and Sgroi \cite{FS2}, established that the $\v$-number $\v(I^k)$ of the powers of $I$ also exhibits eventual linear behavior (see also \cite{BMS24, Conca23, F2023, FS2, FSPack, FSPackA}). The Noetherian property of the Rees algebra $\mathcal{R}(I)=\bigoplus_{k\ge0}I^k$, viewed as a bigraded ring, provides a structural explanation for this asymptotic linearity.
	
	A major open challenge in this area is to describe explicitly the stable set of associated primes $\Ass^\infty I$, as well as the eventual linear functions $\reg\,I^k$ and $\v(I^k)$, for broad classes of ideals arising from combinatorial structures. This paper contributes to this broader effort by investigating a specific and well-studied family of monomial ideals: edge ideals of finite simple graphs.
	
	Let $G$ be a finite simple graph on the vertex set $V(G) = [n] = \{1, \dots, n\}$ and with edge set $E(G)$. The edge ideal of $G$ is the monomial ideal $I(G) \subset S$ generated by the monomials $x_i x_j$ for each edge $\{i, j\} \in E(G)$. The study of the syzygies and homological properties of powers of edge ideals has a rich history in Commutative Algebra, with connections to combinatorial topology and algebraic geometry; see for instance \cite{AB, BDMS, F-HHZ, JS, HHZ2004, LT, NP}. Unlike general monomial ideals, edge ideals often exhibit more structured and predictable behavior. For instance, it is known that there exist monomial ideals $I \subset S$ for which the behavior of $\Ass\,I^k$ can be arbitrarily complicated \cite[Corollary 4.2]{HNTT}. By contrast, Mart\'inez-Bernal, Morey, and Villarreal \cite{MMV} proved that for any edge ideal $I(G)$, the sequence of associated primes satisfies $\Ass\,I(G)^k\subseteq\Ass\,I(G)^{k+1}$ for all $k \ge 1$. Additionally, while a monomial ideal $I \subset S$ with a linear resolution does not guarantee that $I^k$ has a linear resolution for all $k \ge 2$, Herzog, Hibi, and Zheng \cite{HHZ2004} showed that if $I(G)$ has a linear resolution, then $I(G)^k$ also has a linear resolution for every $k \ge 1$.
	
	Building on these insights, our previous work \cite{FQ} introduced a new approach to studying the syzygies of powers of monomial ideals via the lens of homological shift ideals. Given a monomial ideal $I \subset S$, let ${\bf a} = (a_1, \dots, a_n) \in \mathbb{Z}_{\ge 0}^n$, and define ${\bf x^a} = x_1^{a_1} \cdots x_n^{a_n}$. Following \cite{HMRZ021a}, the $i$th \textit{homological shift ideal} of $I$ is  
	$$
	\HS_i(I)\ =\ ( {\bf x^a}\ :\ \Tor_i^S(K, I)_{\bf a} \ne 0).
	$$
	
	Homological shift ideals have been studied extensively in recent literature \cite{Bay019, Bay2023, BJT019, CF1, CF2, F2, FPack1, F-SCDP, FH2023, FQ, HMRZ021a, HMRZ021b, LW, TBR24} and encode critical multigraded data from the minimal free resolution of $I$. Motivated by conjectures from \cite{CF1} and \cite{HMRZ021a}, we introduced in \cite{FQ} the $i$th \textit{homological shift algebra} of $I$ as the $K$-algebra  
	\[
	\HS_i(\mathcal{R}(I)) = \bigoplus_{k \ge 1} \HS_i(I^k).
	\]
	
	These modules generalize the Rees algebra $\mathcal{R}(I)$ (since $S\oplus\HS_0(\mathcal{R}(I)) = \mathcal{R}(I)$) and provide a higher homological framework for understanding asymptotic behaviors of syzygies. In \cite{FQ}, we proved that if $I$ has linear powers, that is, if all powers $I^k$ have linear resolution, then $\HS_i(\mathcal{R}(I))$ is finitely generated over $\mathcal{R}(I)$, leading to the fact that that many invariants of $\HS_i(I^k)$, such as depth, associated primes, regularity and the $\v$-number behave asymptotically well.
	
	Our goal in this paper is to advance the understanding of the asymptotic behavior of syzygies of powers of edge ideals through the framework of homological shift ideals. Below, we summarize the main lines of investigation:
	\begin{enumerate}
		\item[(a)] We prove that the algebra $\HS_1(\mathcal{R}(I(G)))$ is a $\mathcal{R}(I(G))$-module by showing that $\HS_1(I(G)^{k+1})=I(G)\cdot\HS_1(I(G)^k)$ for all $k\ge1$.\smallskip
		\item[(b)] We explore the sets of associated primes of $\HS_i(I(G)^k)$ for all $k\ge1$. Our goal is to describe the behavior of the sets $\Ass\,\HS_i(I(G)^k)$. To achieve this, we formulate Conjecture \ref{Conj:C-A} and we prove it in several significant cases.\smallskip
		\item[(c)] We examine the linearity of the minimal free resolution of $\HS_i(I(G)^k)$ for $k \gg 0$, assuming that $I(G)$ itself has a linear resolution. In support of this, we propose Conjecture~\ref{Conj:C-B} and prove several supporting results.
	\end{enumerate}\medskip
	
	In Section~\ref{sec2}, we study the first homological shift algebra of an edge ideal $I(G)$. Let $I \subset S$ be a monomial ideal. By \cite[Theorem 1.2]{FQ}, we have $\HS_i(I^{k+1}) \subseteq I\cdot\HS_i(I^k)$ for all $k \gg 0$. So, it is natural to inquire whether equality holds. In general, $I \cdot \HS_i(I^k)$ can be much larger than $\HS_i(I^{k+1})$. Indeed, there exist monomial ideals $I$ such that $\HS_i(I^{k+1}) \ne I \cdot \HS_i(I^k)$ for every $k \ge 1$ \cite[Example 1.1]{FQ}. Remarkably, we demonstrate in Corollary~\ref{Cor:HSRI(G)} that $\HS_1(I(G)^{k+1})=I(G)\cdot\HS_1(I(G)^{k})$ for all $k\ge1$. This result stems from Theorem~\ref{Thm:HS1-I(G)k}, where we explicitly describe $\HS_1(I(G)^k)$ for every $k$. We show that $\HS_1(I(G)^k)$ decomposes as the sum of a ``linear part" $\Lin_1(I(G)^k)$ and a ``non-linear part" $I(G)^{\langle k+1 \rangle}$. Hence, we deduce that $\HS_1(\mathcal{R}(I(G)))$ is a $\mathcal{R}(I(G))$-module. Moreover, Theorem~\ref{Thm:Ratliff<>} establishes that $I(G)^{\langle k+1\rangle}:I(G)=I(G)^{\langle k\rangle}$ for all $k\ge2$.
	
	In Section~\ref{sec3}, we turn to the behavior of the sets $\Ass\,\HS_i(I(G)^k)$ as $k$ varies. Based on extensive computations, we conjecture:
	\begin{Conj}\label{Conj:C-A}
		Let $G$ be a finite simple graph. Then, for all $i\ge0$,
		$$\Ass\,\HS_i(I(G)^{\max(1,i)})\,\subseteq\,\Ass\,\HS_i(I(G)^{i+1})\,\subseteq\,\Ass\,\HS_i(I(G)^{i+2})\,\subseteq\,\cdots.$$
	\end{Conj}
	This result is well-known for $i = 0$, first proven by Mart\'inez-Bernal, Morey, and Villarreal \cite{MMV}. In Theorem~\ref{Thm:I(G)-Strong-Persistence}, we prove that Conjecture~\ref{Conj:C-A} holds for $i = 1$. To this end, inspired by \cite{HQ}, we introduce the notion of the $i$th \emph{homological strong persistence property}. In Theorem~\ref{Thm:HomStrongPer}, we show that a monomial ideal $I$ has the $i$th homological strong persistence property if and only if $\HS_i(I^{k+1}):I=\HS_i(I^k)$ for all $k\ge1$. Applying results from Section~\ref{sec2}, we establish this property for $i = 1$ and $I = I(G)$, confirming Conjecture~\ref{Conj:C-A} for $i = 1$. Furthermore, we verified this conjecture computationally for all graphs with up to seven vertices, utilizing the \textit{Macaulay2} \cite{GDS} packages \texttt{HomologicalShiftIdeals} \cite{FPack1} and \texttt{NautyGraphs} \cite{MPNauty}.
	
	Finally, in Section \ref{sec4} we turn to the linearity of powers of $I(G)$. We propose:
	\begin{Conj}\label{Conj:C-B}
		Let $G$ be a finite simple graph such that $I(G)$ has linear resolution. Then, $\HS_i(I(G)^k)$ has linear resolution for all $i\ge0$ and all $k\gg0$.
	\end{Conj}
	For $i = 0$, this result follows from a theorem of Herzog, Hibi, and Zheng \cite{HHZ2004}. Using some results from \cite{FH2023}, we prove in Theorem~\ref{Thm:HS1-I(G)^k-Lin} that Conjecture~\ref{Conj:C-B} holds for $i = 1$. Additionally, we establish Conjecture~\ref{Conj:C-B} for principal Borel edge ideals, initial lexsegment edge ideals, and edge ideals of complete multipartite graphs.
	
	\section{A quick review on the homological shift algebra}\label{sec1}
	
	In this section, we summarize some of the main results from \cite{FQ}.\smallskip
	
	Let $I\subset S$ be a graded ideal and let $\p\in\Ass(I)$ be an associated prime. Recall that the \textit{regularity} of $I$ is defined as $\reg\,I=\max\{j:\ \Tor^S_i(K,I)_{i+j}\ne0\}$, the \textit{$\v_\p$-number} of $I$ is defined as $\v_\p(I)=\min\{\deg(f):\ f\in S_d,\ (I:f)=\p\}$, and the \textit{$\v$-number} of $I$ is defined as $\v(I)=\min\limits_{\p\in\Ass(I)}\v_\p(I)$.\smallskip
	
	Now, let $I\subset S$ be a monomial ideal and denote by $\mathcal{G}(I)$ its minimal monomial generating set. For systematic reasons, we put $I^0=S$.
	
	The $i$th \textit{homological shift algebra} of $I$ is introduced in \cite{FQ} as the $K$-algebra,
	$$
	\HS_i(\mathcal{R}(I))\ =\ \bigoplus_{k\ge1}\HS_i(I^k).
	$$
	
	Notice that $S\oplus\HS_0(\mathcal{R}(I))$ is just the Rees algebra $\mathcal{R}(I)=\bigoplus_{k\ge0}I^k$ of $I$.
	
	For a monomial ideal $I\subset S$, in general it is not possible to structure $\HS_i(\mathcal{R}(I))$ as a $\mathcal{R}(I)$-module, see \cite[Example 1.1]{FQ}. In fact, $\HS_i(\mathcal{R}(I))$ is a $\mathcal{R}(I)$-module if and only if the inclusion $I\cdot\HS_i(I^k)\subseteq\HS_i(I^{k+1})$ holds for all $k\ge1$.
	
	A monomial ideal $I\subset S$ has \textit{linear powers} if $I^k$ has linear resolution for all $k\ge1$. When $I$ has linear powers, we showed in \cite[Theorem 1.4]{FQ} that the inclusion $I\cdot\HS_i(I^k)\subseteq\HS_i(I^{k+1})$ holds for all $k\ge1$, and also that $\HS_i(\mathcal{R}(I))$ is a finitely generated graded $\mathcal{R}(I)$-module.
	
	The next results were proved in \cite[Theorem 1.4]{FQ} and \cite[Theorem 2.1]{FQ}.
	\begin{Theorem}\label{Thm:resume-HS(R(I))}
		Let $I\subset S$ be a monomial ideal with linear powers. Then $\HS_i(\mathcal{R}(I))$ is a finitely generated graded $\mathcal{R}(I)$-module. Hence, we have $\HS_i(I^{k+1})=I\cdot\HS_i(I^k)$ for all $k\gg0$. Furthermore, the following statements hold.
		\begin{enumerate}
			\item[\textup{(a)}] The set $\Ass\,\HS_i(I^k)$ stabilizes: $\Ass\,\HS_i(I^{k+1})=\Ass\,\HS_i(I^{k})$ for all $k\gg0$.
			\item[\textup{(b)}] For all $k\gg0$, we have $\depth\, S/\HS_i(I^{k+1})=\depth\,S/\HS_i(I^k)$.
			\item[\textup{(c)}] For all $k\gg0$, $\reg\,\HS_i(I^k)$ is a linear function in $k$.
			\item[\textup{(d)}] For all $k\gg0$, $\v(\HS_i(I^k))$ is a linear function in $k$.
			\item[\textup{(e)}] For all $k\gg0$ and all $\p\in\Ass\,\HS_i(I^k)$, $\v_\p(\HS_i(I^k))$ is a linear function in $k$.
		\end{enumerate}
	\end{Theorem}
	
	\section{On the structure of $\HS_1(\mathcal{R}(I(G)))$}\label{sec2}
	
	Let $G$ be a finite simple graph on the vertex set $[n]=\{1,\dots,n\}$, with edge set $E(G)$. The \textit{edge ideal} of $G$ is the monomial ideal $I(G)\subset S=K[x_1,\dots,x_n]$ generated by the monomials $x_ix_j\in S$ for which $\{i,j\}\in E(G)$ is an edge of $G$.
	
	Our goal is to describe the first homological shift algebra of an edge ideal. It will turn out that $\HS_1(\mathcal{R}(I(G)))$ is a $\mathcal{R}(I(G))$-module generated in degree one.
	
	For a monomial ideal $J\subset S$, we denote by $J_{\langle d\rangle}$ the ideal generated by all monomials of degree $d$ belonging to $J$. Let $I\subset S$ be a monomial ideal generated in degree $d$, we define the $i$th \textit{homological linear part} of $I$ to be the  ideal, 
	$$
	\Lin_i(I)\ =\ \HS_i(I)_{\langle d+i\rangle}.
	$$
	
	Given a monomial ideal $I\subset S$ with $\mathcal{G}(I)=\{u_1,\dots,u_m\}$ and an integer $k\ge2$, we define the $k$th \textit{non-pure power} of $I$ to be the ideal,
	$$
	I^{\langle k\rangle}\ =\ (u_{i_1}\cdots u_{i_k}\ :\ 1\le i_1\le\dots\le i_k\le m,\ |\{i_1,\dots,i_k\}|>1).
	$$
	By convention, we put $I^{\langle1\rangle}=I$. Notice that $I^{\langle k\rangle}\subset I^k$ for all $k$.
	
	\begin{Theorem}\label{Thm:HS1-I(G)k}
		Let $G$ be a finite simple graph. For all $k\ge1$, we have
		$$
		\HS_1(I(G)^{k})\ =\ \Lin_1(I(G)^{k})+I(G)^{\langle k+1\rangle}.
		$$
	\end{Theorem}
	
	To prove Theorem \ref{Thm:HS1-I(G)k}, we first need the following lemma. Given two monomials $u,v\in S$, we put $u:v=\lcm(u,v)/v$.
	\begin{Lemma}\label{Lem:HS1}
		Let $I\subset S$ be a monomial ideal with $\mathcal{G}(I)=\{u_1,\dots,u_m\}$. Then,
		$$
		\HS_1(I)\ =\ (\lcm(u_i,u_j)\ :\ 1\le i<j\le m).
		$$
	\end{Lemma}
	\begin{proof}
		Let $F_0=S^m$ be the multigraded free $S$-module with basis ${\bf e}_1,\dots,{\bf e}_m$ and with $\Deg({\bf e}_i)=\Deg(u_i)$ for all $i$. Here, by $\Deg(u)=(a_1,\dots,a_n)\in\ZZ_{\ge0}^n$ we denote the multidegree of a monomial $u=x_1^{a_1}\cdots x_n^{a_n}$. Let $\varphi:F_0\rightarrow I$ be the multigraded map defined by $\varphi({\bf e}_i)=u_i$ for all $i$. The kernel $\Ker(\varphi)$ is generated by the relations ${\bf r}_{ij}=(u_i:u_j){\bf e}_j-(u_j:u_i){\bf e}_i$ for $1\le i<j\le m$, and $\Deg({\bf r}_{ij})=\Deg(\lcm(u_i,u_j))$. Let $\{{\bf r}_{i_1j_1},\dots,{\bf r}_{i_\ell j_\ell}\}$ be a minimal generating set of $\Ker(\varphi)$, and let $F_1=S^\ell$ be the multigraded free $S$-module with basis ${\bf f}_1,\dots,{\bf f}_\ell$ with $\Deg({\bf f}_p)=\Deg({\bf r}_{i_pj_p})$ for all $p$. Let $\psi:F_1\rightarrow F_0$ be the multigraded map defined by $\psi({\bf f}_p)={\bf r}_{i_pj_p}$ for all $p$. Then,
		$$
		F_1\xrightarrow{\psi}F_0\xrightarrow{\varphi}I\rightarrow0
		$$
		is the beginning of the minimal multigraded free resolution of $I$. It follows that
		$$
		\HS_1(I)\ =\ (\lcm(u_{i_1},u_{j_1}),\dots,\lcm(u_{i_\ell},u_{j_\ell}))\ \subseteq\ (\lcm(u_i,u_j)\ :\ 1\le i<j\le m).
		$$
		
		To show the opposite inclusion, let $i<j$. If $\{i,j\}=\{i_p,j_p\}$ for some $p$, we have $\lcm(u_i,u_j)\in\HS_1(I)$. Suppose now that ${\bf r}_{ij}$ is not a minimal generator of $\Ker(\varphi)$. Then, there exist monomials $w_1,\dots,w_\ell\in S$ and coefficients $k_1,\dots,k_\ell\in K$ such that
		${\bf r}_{ij}=\sum_{p=1}^\ell k_pw_p{\bf r}_{i_pj_p}$, and $\Deg({\bf r}_{ij})=\Deg(w_p{\bf r}_{i_pj_p})$ for all $p$ with $k_p\ne0$. We have $k_q\ne0$ for some $q$. Hence $\lcm(u_i,u_j)=w_q\lcm(u_{i_q},u_{j_q})\in\HS_1(I)$.
	\end{proof}
	
	Next, we describe the first homological linear part of the powers of an edge ideal.
	\begin{Proposition}\label{Prop:Lin1I(G)}
		Let $G$ be a finite simple graph. Then, for all $k\ge1$,
		\begin{equation}\label{eq:LinI(G)}
			\Lin_1(I(G)^k)=(e_{1}\cdots e_{k}x_p:\ \textstyle e_i\in\mathcal{G}(I(G)),\ x_p(e_1/x_q)\in I(G)\ \textit{with}\ p\ne q).
		\end{equation}
		In particular, $\Lin_1(I(G)^{k+1})=I(G)\cdot\Lin_1(I(G)^k)$ for all $k\ge1$.
	\end{Proposition}
	\begin{proof}
		By (\ref{eq:LinI(G)}), it follows immediately that $\Lin_1(I(G)^{k+1})=I(G)\cdot\Lin_1(I(G)^k)$ for all $k\ge1$. So, it remains to prove equation (\ref{eq:LinI(G)}). Since $I(G)^{k}$ is equigenerated in degree $2k$, the ideal $\Lin_1(I(G)^k)$ is generated by the least common multiples $\lcm(u,v)$ of degree $2k+1$, where $u=e_1\cdots e_k$, $v=f_1\cdots f_k$, and $e_i,f_i\in\mathcal{G}(I(G))$ for all $i$. Since $\deg\lcm(u,v)=2k+1$, we have $\lcm(u,v)=ux_p=vx_t$ for some $p\ne t$.
		
		Now, we prove the statement by induction on $k$. For $k=1$, we have $u=e_1=x_ix_j$ and $v=f_1=x_rx_s$ with $u\ne v$. Since $x_pe_1=x_tf_1$ and $x_p\ne x_t$, we may assume that $x_p=x_r$ and $x_t=x_j$. Then $x_i=x_s$ and so $x_p(e_1/x_t)=f_1\in I(G)$, as desired.
		
		For the inductive step, we may assume that $e_i\ne f_j$ for all $i$ and $j$. Indeed, up to a relabeling, suppose that $e_k=f_k$, and consider the monomials $u'=u/e_k$, $v'=v/f_k$. Then $u',v'\in\mathcal{G}(I(G)^{k-1})$, $\lcm(u,v)=\lcm(u',v')e_k$, and $\deg\lcm(u',v')=2(k-1)+1$. By induction, we have $\lcm(u',v')=e_1\cdots e_{k-1}x_p$ and $x_p(e_1/x_q)\in I(G)$ for some $p\ne q$. Then $\lcm(u,v)=e_1\cdots e_kx_p$ is as described in equation (\ref{eq:LinI(G)}).
		
		Let $e_i\ne f_j$ for all $i$ and $j$. Since $ux_p=vx_t$, we have $v=x_p(u/x_t)=f_1\cdots f_k$. Hence $x_p$ divides some $f_j$, say $f_1$. Then $f_1=x_px_s$ for some $s\ne p$. Since $x_s$ divides $x_p(u/x_t)$, it follows that $x_s$ divides some $e_i$, say for $i=1$. Let $e_1=x_sx_q$. Since $e_1\ne f_1$, we have $p\ne q$ and $f_1=x_p(e_1/x_q)\in I(G)$, as desired.
	\end{proof}
	
	We are now in the position to prove Theorem \ref{Thm:HS1-I(G)k}.
	\begin{proof}[Proof of Theorem \ref{Thm:HS1-I(G)k}]
		Note that $\Lin_1(I(G)^k)\subseteq\HS_1(I(G)^k)$ by definition. We claim that $I(G)^{\langle k+1\rangle}\subseteq\HS_1(I(G)^k)$. Indeed, let $w=e_1\cdots e_{k+1}\in\mathcal{G}(I(G)^{\langle k+1\rangle})$ with $e_i\in\mathcal{G}(I(G))$ for all $i$, and $e_1\ne e_2$. Let $u=e_1\cdot e_3\cdots e_{k+1}$ and $v=e_2\cdot e_3\cdots e_{k+1}$. Then $u,v\in\mathcal{G}(I(G)^k)$, and so $\lcm(u,v)=\lcm(e_1,e_2)\cdot e_3\cdots e_{k+1}\in\HS_1(I(G)^k)$ by Lemma \ref{Lem:HS1}. Since $\lcm(u,v)$ divides $w$, it follows that $w\in\HS_1(I(G)^k)$. Hence
		$$
		\Lin_1(I(G)^k)+I(G)^{\langle k+1\rangle}\ \subseteq\ \HS_1(I(G)^k).
		$$
		
		For the opposite inclusion, recall that by Lemma \ref{Lem:HS1}, $\HS_1(I(G)^k)$ is generated by the least common multiples $\lcm(u,v)$ with $u,v\in\mathcal{G}(I(G)^k)$ and $u\ne v$. Let $w=\lcm(u,v)\in\mathcal{G}(\HS_1(I(G)^k))$ be such a minimal generator. Since $I(G)^k$ is equigenerated in degree $2k$, we have $\deg(w)\ge2k+1$. If $\deg(w)=2k+1$, we have $w\in\Lin_1(I(G)^k)$. Suppose now $\deg(w)\ge 2k+2$. Proceeding by induction on $k$, we will show that $w\in I(G)^{\langle k+1\rangle}$ and this will finish the proof. We can write $w=fu=gv$ where $f=v:u$ and $g=u:v$.  Then $\deg(f)=\deg(g)\ge 2$. Let $u=e_1\cdots e_k$ and $v=e_1'\cdots e_k'$, with $e_i,e_i'\in\mathcal{G}(I(G))$ for all $i$.  Let $f=x_{i_1}\cdots x_{i_d}$. Since $fu=gv$ and $\gcd(f,g)=1$, we have that $x_{i_1}$ and $x_{i_2}$ divide $v$.
		
		Suppose first $k=1$. Since $\deg(w)\ge4$, it follows that $w=uv$ and $\gcd(u,v)=1$. Therefore $u=e_1\ne e_1'=v$ and so $w\in I(G)^{\langle 2\rangle}$, as desired.
		
		Now, suppose $k>1$. We may assume that $e_i\ne f_j$ for all $i$ and $j$. Indeed, suppose that $e_k=f_k$, and consider the monomials $u'=u/e_k$, $v'=v/f_k$. Then $u',v'\in\mathcal{G}(I(G)^{k-1})$, $\lcm(u,v)=\lcm(u',v')e_k$, and $\deg\lcm(u',v')\ge2k$. We claim that $\lcm(u',v')\in\mathcal{G}(\HS_1(I(G)^{k-1}))$. Otherwise, there exist $u'',v''\in\mathcal{G}(I(G)^{k-1})$ such that $w''=\lcm(u'',v'')\in\mathcal{G}(\HS_1(I(G)^{k-1}))$ strictly divides $w'=\lcm(u',v')$. Then $u''e_k,v''e_k\in\mathcal{G}(I(G)^k)$, and $\lcm(u''e_k,v''e_k)=w''e_k$ strictly divides $w'e_k=w$. This is absurd because $w\in\mathcal{G}(\HS_1(I(G)^k))$. Hence $w'\in\mathcal{G}(\HS_1(I(G)^{k-1}))$. Since $\deg(w')\ge 2(k-1)+2$, it follows by induction that $w'\in\mathcal{G}(I(G)^{\langle k\rangle})$. Finally, we have $w=w'e_k\in\mathcal{G}(I(G)^{\langle k+1\rangle})$, as desired.
		
		Now, assume that $e_i\ne f_j$ for all $i$ and $j$. We distinguish the two possible cases.\smallskip
		
		\textsc{Case 1.} Suppose that $e_s'=x_{i_1}x_{i_2}$ for some $s$. Let $w'=ue_s'=e_1\cdots e_ke_s'$. We claim that $w'\in I(G)^{\langle k+1\rangle}$. Otherwise $e_1=\dots=e_k=e_s'=x_{i_1}x_{i_2}$, and then $(x_{i_1}x_{i_2})^{k+1}$ divides $fu=gv$. This is not possible. Indeed, since $\gcd(f,g)=1$ and $x_{i_1}x_{i_2}$ divides $f$, then $(x_{i_1}x_{i_2})^{k+1}$ should divide $v$, which is not possible because $\deg(v)=2k$ while $\deg((x_{i_1}x_{i_2})^{k+1})=2(k+1)$. Hence $w'\in I(G)^{\langle k+1\rangle}$, and so $w\in I(G)^{\langle k+1\rangle}$.\smallskip
		
		\textsc{Case 2.} Suppose that $x_{i_1}x_{i_2}\ne e_s'$ for all $s$. Up to a relabeling, we may assume that $e_1'=x_{i_1}x_{j_1}$ and $e_2'=x_{i_2}x_{j_2}$. We have $j_1\ne i_2$, otherwise $e_1'=x_{i_1}x_{i_2}$ against the assumption. Similarly $j_2\ne i_1$. If $x_{j_1}$ divides $f$, then $x_{i_1}x_{j_1}=e_1'$ divides $f$ and we can apply the first case. We can proceed similarly if $x_{j_2}$ divides $f$. Hence, we may assume that $x_{j_1}$ and $x_{j_2}$ do not divide $f$. Then $x_{j_1}x_{j_2}$ divides $u$. If $x_{j_1}x_{j_2}=e_s$ for some $s$, then $x_{i_1}x_{i_2}u=e_1'e_2'(e_1\cdots e_{s-1}e_{s+1}\cdots e_k)\in I(G)^{\langle k+1\rangle}$ because $e_1'\ne e_2'$. Then $w\in I(G)^{\langle k+1\rangle}$. Suppose now, up to a relabeling, that $e_1=x_{j_1}x_{\ell_1}$ and $e_2=x_{j_2}x_{\ell_2}$. Since $e_1\ne e_1'$ and $e_2\ne e_2'$ by our assumption, we obtain that $\ell_1\ne i_1$ and $\ell_2\ne i_2$. Then, we have $e_1'=x_{i_1}(e_1/x_{\ell_1})\ne e_1$. Hence $ux_{i_1}\in\Lin_1(I(G)^k)$ divides $fu=w$ against the fact that $w\in\mathcal{G}(\HS_1(I(G)^k))$. This concludes the proof.
	\end{proof}
	\begin{Corollary}\label{Cor:HSRI(G)}
		Let $G$ be a finite simple graph. For all $k\ge1$,
		$$
		\HS_1(I(G)^{k+1})=I(G)\cdot\HS_1(I(G)^k).
		$$
		In particular, $\HS_1(\mathcal{R}(I(G)))$ is a $\mathcal{R}(I(G))$-module generated in degree one.
	\end{Corollary}
	\begin{proof}
		By Proposition \ref{Prop:Lin1I(G)}, we have $\Lin_1(I(G)^{k+1})=I(G)\cdot\Lin_1(I(G)^k)$ for all $k\ge1$. Moreover, it is clear that $I(G)^{\langle k+2\rangle}=I(G)\cdot I(G)^{\langle k+1\rangle}$ for all $k\ge1$. Therefore, applying Theorem \ref{Thm:HS1-I(G)k}, for all $k\ge1$ we have
		\begin{align*}
			\HS_1(I(G)^{k+1})\ &=\ \Lin_1(I(G)^{k+1})+I(G)^{\langle k+2\rangle}\\
			&=\ I(G)\cdot\Lin_1(I(G)^{k})+I(G)\cdot I(G)^{\langle k+1\rangle}\ = \ I(G)\cdot \HS_1(I(G)^k),
		\end{align*}
		as desired. Since this equality holds for all $k\ge1$, it follows that $\HS_1(\mathcal{R}(I(G)))$ is a finitely generated bigraded $\mathcal{R}(I(G))$-module, which is generated by its first degree component $\HS_1(I(G))$.
	\end{proof}
	
	It is an open question whether we have $\HS_i(I(G)^{k+1})=I(G)\cdot\HS_i(I(G)^k)$ for all $k\gg0$ and all graphs $G$, when $i>1$. Moreover, for $i>1$, it is not clear whether $\HS_i(\mathcal{R}(I(G)))$ is a $\mathcal{R}(I(G))$-module.
	
	\begin{Theorem}\label{Thm:Ratliff<>}
		Let $G$ be a finite simple graph. Then, for all $k\ge2$,
		$$
		I(G)^{\langle k+1\rangle}:I(G)\ =\ I(G)^{\langle k\rangle}.
		$$
	\end{Theorem}
	\begin{proof}
		It is clear that $I(G)^{\langle k\rangle}\subseteq I(G)^{\langle k+1\rangle}:I(G)$ for all $k\ge2$. Conversely, let $u\in I(G)^{\langle k+1\rangle}:I(G)$ be a monomial. Since we have $I(G)^{\langle k+1\rangle}\subset I(G)^{k+1}$, we obtain $u\in I(G)^{\langle k+1\rangle}:I(G)\subseteq I(G)^{k+1}:I(G)=I(G)^{k}$,
		where the last equality follows from \cite[Lemma 2.12]{MMV}. Hence $u=e_1\cdots e_kf$ where $f\in S$ is a monomial and $e_i \in \mathcal{G}(I(G))$ for all $i$ . If $e_p\ne e_q$ for some $1\le p<q\le k$, then $e_1\cdots e_k\in I(G)^{\langle k\rangle}$ and so $u\in I(G)^{\langle k\rangle}$, as required. Suppose $e_1=\dots=e_k$. Then $u=e^kf$ where $e=e_1$. If $f$ is divisible by some $v\in\mathcal{G}(I(G))$ with $v\ne e$, then $u\in I(G)^{\langle k\rangle}$. Therefore, we may assume that $u=e^\ell g$ with $\ell\ge k$ and $g\notin I(G)$. Since $u\in I(G)^{\langle k+1\rangle}:I(G)$, we have $ue=e^{\ell+1}g\in I(G)^{\langle k+1\rangle}$. Hence
		\begin{equation}\label{eq:ue}
			ue\ =\ e^{\ell+1}g\ =\ f_1\cdots f_{k+1} h,
		\end{equation}
		where $f_1\cdots f_{k+1}\in I(G)^{\langle k+1\rangle}$, $f_i\in\mathcal{G}(I(G))$ for all $i$, $f_1\ne f_2$, and $h\in S$. Since $f_1\ne f_2$, at least one between $f_1$ and $f_2$ is different from $e$, say $f_1\ne e$. Let $f_1=x_px_q$. Since $f_1\ne e$, we may assume that $x_p$ does not divide $e$. From (\ref{eq:ue}) it follows that $x_p$ divides $g$. Since $g\notin I(G)$, it follows that $x_q$ divides $e$. Let $e=x_qx_r$ with $r\ne p$. Hence $u=e^{\ell-1}f_1 h$ with $h=x_r(g/x_p)$. Since $\ell-1\ge k-1\ge 1$ and $e\ne f_1$, it follows that $e^{\ell-1}f_1\in I(G)^{\langle k\rangle}$, and so $u\in I(G)^{\langle k\rangle}$, as desired.
	\end{proof}
	
	We remark that for $k=1$, it is not true that $I(G)^{\langle 2\rangle}:I(G)=I(G)$. Indeed, it is clear that $e^2\notin I(G)^{\langle 2\rangle}$ for all $e\in\mathcal{G}(I(G))$. It is unclear whether the equality $I^{\langle k+1\rangle}:I=I^{\langle k\rangle}$ holds for all $k\gg0$ and all monomial ideals $I\subset S$.
	
	\section{The homological persistence properties}\label{sec3}
	
	In this section, our goal is to understand the behavior of the sets $\Ass\,\HS_i(I(G)^k)$ for a fixed $i$, as $k$ varies. For this aim, we introduce the so-called \textit{homological persistence properties}, which we define below.
	
	Let $I\subset S$ be a monomial ideal. By $V(I)$ we denote the set of prime ideals of $S$ containing $I$. Let $P\subset S$ be a monomial prime ideal, $S(P)$ be the polynomial ring in the variables which generate $P$ and $S_P$ be the localization of $S$ with respect to $P$. We recall that the \textit{monomial localization} of $I$ is the monomial ideal $I(P)$ of $S(P)$ obtained from $I$ by applying the substitutions $x_i\mapsto 1$ for $x_i\notin P$. Note that $I(P)$ is a monomial ideal in $S(P)$ and $I(P)S_P=IS_P$.
	
	For later use, we notice some basic facts. Let $I,J,L\subset S$ be monomial ideals and let $P\subset S$ be a monomial prime ideal. Then $I:J$ is again a monomial ideal and $(I:J)(P)=I(P):J(P)$. Suppose that $L\subseteq I:J$ and consider the $S$-module $M=(I:J)/L$. Then $M$ is multigraded, and we may replace the localization $M_P$ with the $S(P)$-module $N=(I:J)(P)/L(P)=(I(P):J(P))/L(P)$, preserving the multigraded structure. Furthermore, $P\in\Ass_SI$ if and only if $P\in\Ass_{S(P)}I(P)$. We will use freely these facts. When there is no risk of confusion about the ambient ring, we will simply write $\Ass\, I$ instead of $\Ass_SI$. Therefore, for monomial ideals, one can replace ordinary localization with monomial localization, with the advantage of preserving the multigraded structure.\smallskip
	
	Inspired by \cite{HQ}, we introduce the following concepts which generalize the persistence and strong persistence properties. Let $I\subset S$ be a monomial ideal. We say that $I$ satisfies the \textit{$i$th homological persistence property} if
	$$
	\Ass\,\HS_i(I)\subseteq\Ass\,\HS_i(I^2)\subseteq\Ass\,\HS_i(I^3)\subseteq\cdots.
	$$
	
	Whereas, we say that $I$ satisfies the \textit{$i$th homological strong persistence property} if for all monomial prime ideals $P\in V(I)$, all integers $k>0$, and any monomial $f\in(\HS_i(I^k):P)(P)\setminus\HS_i(I^k)(P)$, there exists a monomial $g\in I(P)$ such that $g\HS_i(I^k)(P)\subseteq\HS_i(I^{k+1})(P)$ and $fg\notin\HS_i(I^{k+1})(P)$.\smallskip
	
	Notice that in the previous definition, the condition $g\HS_i(I^k)(P)\subseteq\HS_i(I^{k+1})(P)$ is automatically satisfied if $i=0$. Hence, when $i=0$, the above properties are just the persistence and the strong persistence property  introduced in \cite{HQ}.
	
	\begin{Proposition}\label{Prop:ithSPP}
		The $i$th homological strong persistence property implies the $i$th homological persistence property.
	\end{Proposition}
	\begin{proof}
		Let $P\in\Ass_S\HS_i(I^k)$. Then $P\in\Ass_{S(P)}\HS_i(I^k)(P)$ and $P$ is the maximal ideal of the ring $S(P)$. Therefore, there exists a monomial $f\in S(P)$ such that $\HS_i(I^k)(P):f=P$. Hence $f\in(\HS_i(I^k):P)(P)\setminus\HS_i(I^k)(P)$. By assumption, we can find a monomial $g\in I(P)$ such that $g\HS_i(I^k)(P)\subseteq\HS_i(I^{k+1})(P)$ and $fg\notin\HS_i(I^{k+1})(P)$. Hence $P\subseteq\HS_i(I^{k+1})(P):fg$. Since $fg\notin\HS_i(I^{k+1})(P)$ and $P$ is the maximal ideal of $S(P)$, we conclude that $\HS_i(I^{k+1})(P):fg=P$. This shows that $P\in\Ass_{S(P)}\HS_i(I^{k+1})(P)$, and so $P\in\Ass_S\HS_i(I^{k+1})$.
	\end{proof}
	
	The next result generalizes \cite[Theorem 1.3]{HQ}, and characterizes algebraically the homological strong persistence properties.
	
	\begin{Theorem}\label{Thm:HomStrongPer}
		A monomial ideal $I\subset S$ satisfies the $i$th homological strong persistence property if and only if $\HS_i(I^{k+1}):I=\HS_i(I^k)$ for all $k\ge1$.
	\end{Theorem}
	\begin{proof}
		Assume that $\HS_i(I^{k+1}):I=\HS_i(I^k)$ for all $k\ge1$. Let $P\in V(I)$ be a monomial prime ideal. Then $(\HS_i(I^{k+1}):I)(P)=\HS_i(I^k)(P)$ for all $k\ge1$. Let $f\in(\HS_i(I^{k}):P)(P)\setminus\HS_i(I^k)(P)$ be a monomial. Suppose by contradiction that for all $g\in I(P)$ either $g\cdot\HS_i(I^k)(P)\not\subseteq\HS_i(I^{k+1})(P)$ or $fg\in\HS_i(I^{k+1})(P)$. The first condition never occurs because $\HS_i(I^{k+1})(P):I(P)=\HS_i(I^k)(P)$ for all $k$. Hence $fg\in\HS_i(I^{k+1})(P)$ for all $g\in I(P)$. Thus $f\in\HS_i(I^{k+1})(P):I(P)=\HS_i(I^k)(P)$, which is absurd.
		
		Conversely, suppose that $I$ satisfies the $i$th homological strong persistence property, but  $\HS_i(I^{k+1}):I\ne\HS_i(I^k)$ for some integer $k\ge1$. Since the $S$-module $M=(\HS_i(I^{k+1}):I)/\HS_i(I^k)$ is non-zero and multigraded, we can choose a minimal monomial prime ideal $P$ in the support of $M$. As explained before, we may replace the $S_P$-module $M_P$ with the $S(P)$-module $N=(\HS_i(I^{k+1}):I)(P)/\HS_i(I^k)(P)$. Since $P$ is a minimal prime in the support of $M$, the $S(P)$-module $N$ is of finite length. Therefore, we can find a monomial $\widetilde{f}=f+\HS_i(I^{k})(P)\in N$ for which $\widetilde{f}P=0$ in $N$. That is, $f\in(\HS_i(I^{k}):P)(P)\setminus\HS_i(I^k)(P)$. Then, by the assumption, there exists a monomial $g\in I(P)$ such that $g\HS_i(I^k)(P)\subseteq\HS_i(I^{k+1})(P)$ and $fg\notin\HS_i(I^{k+1})(P)$. This is absurd, because $f\in(\HS_i(I^{k+1}):I)(P)$.
	\end{proof}
	
	As a consequence, we have
	\begin{Corollary}\label{Cor:HomRatliff}
		If a monomial ideal $I\subset S$ satisfies the $i$th homological strong persistence property, then $\HS_i(I^{k+1})=I\cdot\HS_i(I^k)=I*\HS_i(I^k)$ for all $k\gg0$.
	\end{Corollary}
	\begin{proof}
		By Theorem \ref{Thm:HomStrongPer} and the assumption, $\HS_i(I^{k+1}):I=\HS_i(I^k)$ for all $k\ge1$. Hence $I\cdot\HS_i(I^k)\subseteq\HS_i(I^{k+1})$ for all $k\ge1$. Conversely, by \cite[Theorem 1.2]{FQ} we have $\HS_i(I^{k+1})\subseteq I\cdot\HS_i(I^k)$ for all $k\gg0$.
	\end{proof}
	
	In \cite{Rat}, Ratliff showed that $I^{k+1}:I=I^k$ for $k\gg0$, if $I$ is an ideal of a Noetherian ring $R$. In particular, if $I\subset S$ is a monomial ideal, then $\HS_0(I^{k+1}):I=\HS_0(I^k)$ for all $k\gg0$. This property is no longer valid for the higher homological shift ideals. Indeed, let $I=(x^2,y^2,xyz)\subset S=K[x,y,z]$. It is shown in \cite[Example 1.3]{FQ} that $\HS_2(I^{k+1})\ne I\cdot\HS_2(I^k)$ for all $k\ge1$. By Corollary \ref{Cor:HomRatliff}, $I$ does not satisfies the second homological strong persistence property. However, $I$ satisfies the second homological persistence property. Indeed, it was shown in \cite[Example 1.3]{FQ} that
	$$
	\HS_2(I^k)=(x^py^q\ :\ p+q=2k\ \textit{and}\ p,q\ \textit{are odd}).
	$$
	Hence, all monomials $u\in\mathcal{G}(\HS_2(I^k))$ are divided by $xy$. Using \cite[Proposition 5.1]{FS2} it follows that $\Ass\,\HS_2(I)=\{(x),(y)\}$ since $\HS_2(I)=(xy)$ is a principal ideal, and $\Ass\,\HS_2(I^k)=\{(x),(y),(x,y)\}$ for $k>1$, because $\HS_2(I^k)$ is not principal. Hence, the ideal $I$ satisfies the second homological persistence property.\smallskip
	
	Let $I\subset S$ be an ideal. By $\Min(I)$ we denote the set of associated primes of $I$ which are minimal with respect to the inclusion.
	\begin{Proposition}\label{prop:Min}
		Let $I\subset S$ be a monomial ideal satisfying the $i$th homological strong persistence property and such that $\HS_i(I^{k+1})=I\cdot\HS_i(I^k)$ for all $k\ge1$. Then $\Min\,\HS_i(I^k)=\Min\,\HS_i(I)$ for all $k\ge1$.
	\end{Proposition}
	\begin{proof}
		By the assumption, we have $\HS_i(I^{k})=I\cdot\HS_i(I^{k-1})=\dots=I^{k-1}\cdot\HS_i(I)$ for all $k\ge2$. Let $P\in\Min\,\HS_i(I^k)$ with $k\ge2$. Then, for all $u\in\mathcal{G}(I)$ we have
		$$
		(u^{k-1})\cdot\HS_i(I)\ \subseteq\ I^{k-1}\cdot\HS_i(I)=\HS_i(I^k)\ \subseteq\ P.
		$$
		
		Suppose there exists $u\in\mathcal{G}(I)\setminus P$. Since $P$ is a prime ideal, it follows that $\HS_i(I)\subseteq P$. Hence, there exists $Q\in\Min\,\HS_i(I)$ such that $\HS_i(I)\subseteq Q\subseteq P$. Since $I$ satisfies the $i$th homological strong persistence property, Proposition \ref{Prop:ithSPP} implies $Q\in\Min\,\HS_i(I)\subseteq\Ass\,\HS_i(I)\subseteq\Ass\,\HS_i(I^k)$. Thus $P=Q$ and so $P\in\Min\,\HS_i(I)$.
		
		Otherwise, suppose that $u\in P$ for all $u\in\mathcal{G}(I)$. Recall that by Lemma \ref{Lem:HS1} we have $\HS_1(I)=(\lcm(u,v):\ u,v\in\mathcal{G}(I),u\ne v)$. Since we obviously have the inclusions $\HS_i(I)\subseteq\HS_{i-1}(I)\subseteq\dots\subseteq\HS_1(I)$, it follows that $\HS_i(I)\subseteq\HS_1(I)\subseteq P$. Arguing as before, we then have $P\in\Min\,\HS_i(I)$.
		
		Hence, we have shown that $\Min\,\HS_i(I^k)\subseteq\Min\,\HS_i(I)$ for all $k$. Since by Proposition \ref{Prop:ithSPP} we have the inclusions $\Min\,\HS_i(I)\subseteq\Ass\,\HS_i(I)\subseteq\Ass\,\HS_i(I^k)$, we conclude that $\Min\,\HS_i(I^k)=\Min\,\HS_i(I)$ for all $k$.
	\end{proof}
	
	Now, we turn our attention to edge ideals. Due to several computational evidence, we expect that
	\begin{Conjecture}\label{Conj:HS-Ass-I(G)}
		Let $G$ be a finite simple graph. Then, for all $i\ge0$, $$\Ass\,\HS_i(I(G)^{\max(1,i)})\,\subseteq\,\Ass\,\HS_i(I(G)^{i+1})\,\subseteq\,\Ass\,\HS_i(I(G)^{i+2})\,\subseteq\,\cdots.$$
	\end{Conjecture}
	
	In support of this conjecture we have
	\begin{Theorem}\label{Thm:I(G)-Strong-Persistence}
		Let $G$ be a finite simple graph. Then, $I(G)$ satisfies the $0$th and the first homological strong persistence properties. In particular, $I(G)$ satisfies the $0$th and the first homological persistence properties.
	\end{Theorem}
	
	For the proof of the theorem, we need the following two preliminary lemmas.
	\begin{Lemma}\label{Lem:ee'}
		Let $G$ be a finite simple graph. Let $e_1,\dots,e_k\in\mathcal{G}(I(G))$ and $f\in S$ be a monomial such that $f\notin I(G)$ and such that the following condition is satisfied.
		\begin{enumerate}
			\item[$(*)$] For any variable $x_p$ dividing $f$, any variable $x_q\ne x_p$ dividing some $e_j$, we have $x_p(e_j/x_q)\notin I(G)$.
		\end{enumerate}
		Then, for all $i$, we have $e_i\cdot(e_1\cdots e_kf)\notin I(G)^{\langle k+2\rangle}$. 
	\end{Lemma}
	\begin{proof}
		Suppose for a contradiction that $e_i\cdot(e_1\cdots e_kf)\in I(G)^{\langle k+2\rangle}$. Up to a relabeling, we may assume $i=k$, and for convenience, we put $e_{k+1}=e_k$. Then, we can write $e_1\cdots e_{k+1}f=e_1'\cdots e_{k+2}'g$ with $e_{i}'\in\mathcal{G}(I(G))$, $|\{e_1',\dots,e_{k+2}'\}|>1$, and $g\in S$ is a suitable monomial. We may assume, up to a relabeling, that $e_1'=e_{j_1},\dots,e_{\ell}'=e_{j_\ell}$ for some $\ell\ge1$, distinct integers $j_1,\dots,j_\ell$, and that $e_j'\ne e_i$ for all $\ell+1\le j\le k+2$ and all $i\in\{1,\dots,k+1\}\setminus\{j_1,\dots,j_\ell\}$. Then, from the equation $e_1\cdots e_{k+1}f=e_{1}'\cdots e_{k+2}'g$, canceling the monomials $e_1'=e_{j_1},\dots,e_{\ell}'=e_{j_\ell}$, we obtain that
		\begin{equation}\label{eq:ee'}
			(\prod_{\substack{s=1,\dots,k+1\\ s\ne j_1,\dots,j_\ell}}e_s)f\ =\ e_{\ell+1}'\cdots e_{k+2}'g.
		\end{equation}
		We claim that $\ell<k+1$. Otherwise, $f=e_{k+2}'g\in I(G)$ against the assumption.
		
		Since
		$$
		\deg(\prod_{\substack{s=1,\dots,k+1\\ s\ne j_1,\dots,j_\ell}}e_s)=2(k-\ell)+2\ <\ 2(k-\ell)+4=\deg(e_{\ell+1}'\cdots e_{k+2}'),
		$$
		it follows that some variable $x_p$ dividing $f$ divides $e_{j}'$ for some $j$, say $j=k+2$. Now, let $e_{k+2}'=x_px_q$. Since $f\notin I(G)$, then $x_q$ does not divide $f$. Hence, equation (\ref{eq:ee'}) implies that $x_q$ divides some $e_s$. Let $e_s=x_qx_r$. Since $e_s\ne e_{k+2}'$, we have $r\ne p$. But then we would have $x_p(e_s/x_r)=e_{k+2}'\in I(G)$ against the condition $(*)$. This contradiction shows that $e_i(e_1\cdots e_k)f\notin I(G)^{\langle k+2\rangle}$.
	\end{proof}
	
	\begin{Lemma}\label{Lem:ee'1}
		Let $G$ be a finite simple graph. Let $e_1,\dots,e_k\in\mathcal{G}(I(G))$ and $f\in S$ be a monomial such that $f\notin I(G)$ and such that the following condition  is satisfied.
		\begin{enumerate}
			\item[$(*)$] For any variable $x_p$ dividing $f$, any variable $x_q\ne x_p$ dividing some $e_j$, we have $x_p(e_j/x_q)\notin I(G)$.
		\end{enumerate}
		Then, for all $i$, we have $e_i\cdot(e_1\cdots e_kf)\notin\Lin_1(I(G)^{k+1})$. 
	\end{Lemma}
	\begin{proof}
		Suppose for a contradiction that $u=e_i\cdot(e_1\cdots e_kf)\in\Lin_1(I(G)^{k+1})$. Up to a relabeling, we may assume $i=k$, and for convenience, we put $e_{k+1}=e_k$. Then, we can write $e_1\cdots e_{k+1}f=e_1'\cdots e_{k+1}'x_pg$ with $e_{i}'\in\mathcal{G}(I(G))$, $x_p(e_1'/x_q)\in I(G)$ for some $x_p\ne x_q$, and $g\in S$ is a suitable monomial. We distinguish two cases now.\smallskip
		
		\textsc{Case 1.} Suppose, up to a relabeling, that $e_{1}'=e_{j_1},\dots,e_{\ell}'=e_{j_\ell}$ for some integer $\ell$, distinct integers $j_1,\dots,j_\ell$, and $e_j'\ne e_i$ for all $\ell+1\le j\le k+1$ and all $i\in\{1,\dots,k+1\}\setminus\{j_1,\dots,j_{\ell}\}$. We have $\ell<k+1$, otherwise $f=x_pg$ and then the condition $(*)$ would be violated, because  $x_p(e_{j_1}/x_q)=x_p(e_1'/x_q)\in I(G)$ with $x_p\ne x_q$ and $x_p$ divides $f$. Hence $\ell<k+1$. Let $v=u/(e_{j_1}\cdots e_{j_\ell}f)$. If $v=e_{\ell+1}'\cdots e_{k+1}'$, then $e_{1}\cdots e_{k+1}=e_1'\cdots e_{k+1}'$. Hence, $f=x_pg$ and so $x_p$ divides $f$. But then the condition $(*)$ would be violated. Therefore $v\ne e_{\ell+1}'\cdots e_{k+1}'$. Since $vf=e_{\ell+1}'\cdots e_{k+1}'x_pg$, it follows that we can find a variable $x_r$ that divides $f$ and $e_j'$ for some $j$, say $j=\ell+1$. Write $e_{\ell+1}'=x_rx_s$. Since $f\notin I(G)$, it follows that $x_s$ does not divide $f$. Hence $x_s$ divides $v$, and so $x_s$ divides some $e_h$ with $h\in\{1,\dots,k+1\}\setminus\{j_1,\dots,j_\ell\}$. Let $e_h=x_sx_t$. Since $e_h\ne e_{\ell+1}'$, we have $x_t\ne x_r$. But then the condition $(*)$ would be violated because $x_r(e_h/x_t)\in I(G)$ with $x_r\ne x_t$ and $x_r$ divides $f$. This contradiction shows that in this case $e_i\cdot(e_1\cdots e_kf)$ does not belong to $\Lin_1(I(G)^{k+1})$.\smallskip
		
		\textsc{Case 2.} Suppose now that $e_1'\ne e_j$ for all $j=1,\dots,k+1$. We may also assume that $x_p(e_1'/x_q)\ne e_j$ for all $j=1,\dots,k+1$. Otherwise, we may consider the equation $ue_1=e_1''e_2'\cdots e_{k+1}'x_qg$ with $e_1''=x_p(e_1'/x_q)$ and argue as in Case 1. Up to a relabeling, we have that $e_2'=e_{j_2},\dots,e_\ell'=e_{j_\ell}$ for some integer $\ell\le k+1$, distinct integers $j_1,\dots,j_\ell$, and $e_j'\ne e_i$ for all $j=\ell+1,\dots,k+1$ and all $i\in\{1,\dots,k+1\}\setminus\{j_1,\dots,j_\ell\}$. Let $w=u/(e_{j_2}\cdots e_{j_\ell}f)$. We obtain that
		\begin{equation}\label{eq:ee'1}
			wf\ =\ (\prod_{\substack{s=1,\dots,k+1\\ s\ne j_2,\dots,j_\ell}}e_s)f\ =\ e_1'e_{\ell+1}'\cdots e_{k+1}'x_pg.
		\end{equation}
		We may assume that $w\ne e_1'e_{\ell+1}'\cdots e_{k+1}'$. If otherwise $w=e_1'e_{\ell+1}'\cdots e_{k+1}'$, we replace $e_1'$ with $e_1''=x_p(e_1'/x_q)$ and consider the equation $wf=e_1''e_{\ell+1}'\cdots e_{k+1}'x_qg$. Now, having that $w\ne e_1'e_{\ell+1}'\cdots e_{k+1}'$, and since these two monomials have the same degree $2(k-\ell)+4$, it follows from equation (\ref{eq:ee'1}) that some variable $x_r$ dividing $f$ divides $e_h'$ for some $h\in\{1,\ell+1,\ell+2,\dots,k+1\}$. Let $e_h'=x_rx_s$. Since $f\notin I(G)$, it follows that $x_s$ does not divide $f$. Hence, again from (\ref{eq:ee'1}) it follows that $x_s$ divides $e_d$ for some $e_d$ appearing in $w$. Let $e_d=x_sx_t$. Since $e_d\ne e_{h}'$ we have $x_t\ne x_r$. Finally, condition $(*)$ is violated because $x_r$ divides $f$, $x_r\ne x_t$ and $x_r(e_d/x_t)\in I(G)$. This contradiction concludes the proof.
	\end{proof}
	
	Finally, we are in the position to prove Theorem \ref{Thm:I(G)-Strong-Persistence}.
	\begin{proof}[Proof of Theorem \ref{Thm:I(G)-Strong-Persistence}]
		Since $\HS_0(I(G)^k)=I(G)^k$ for all $k$, \cite[Lemma 2.12]{MMV} implies that $I(G)$ satisfies the $0$th homological strong persistence property.\smallskip
		
		Now, we prove that $I(G)$ satisfies the first homological strong persistence property. From Theorem \ref{Thm:HS1-I(G)k} we have $\HS_1(I(G)^{k+1})=I(G)\cdot\HS_1(I(G)^k)$ for all $k\ge1$. Hence $\HS_1(I(G)^k)\subseteq\HS_1(I(G)^{k+1}):I(G)$. Thus, it remains to prove the opposite inclusion. Let $u\in\HS_{1}(I(G)^{k+1}):I(G)$ be a monomial. By \cite[Lemma 2.12]{MMV}, since $\HS_1(I(G)^{k+1})\subseteq I(G)^{k+1}$, we obtain that
		$$
		u\in\HS_1(I(G)^{k+1}):I(G)\ \subseteq\ I(G)^{k+1}:I(G)\ =\ I(G)^{k},
		$$
		So, we can write $u=e_1\cdots e_\ell f$, where $\ell\ge k$, $e_i\in\mathcal{G}(I(G))$ for all $i$ and $f\notin I(G)$.
		
		If for some variable $x_p$ dividing $f$, we can find a variable $x_q$ and an integer $j$ such that $x_p(e_j/x_q)\ne e_j$ and $x_p(e_j/x_q)\in I(G)$, then it follows that $$e_1\cdots e_\ell x_p\in\Lin_1(I(G)^{\ell})\subseteq\Lin_1(I(G)^k)\subseteq\HS_1(I(G)^k)$$ by Theorem \ref{Thm:HS1-I(G)k}. Hence $u\in\HS_1(I(G)^k)$ too, as desired.
		
		Suppose now that the following condition $(*)$ is satisfied: For any variable $x_p$ dividing $f$, any variable $x_q\ne x_p$ dividing some $e_j$, we have $x_p(e_j/x_q)\notin I(G)$. We distinguish two cases now.\smallskip
		
		\textsc{Case 1.} Let $\ell>k$. We claim that $e_i\ne e_j$ for some $i\ne j$. Once we acquire this claim, we have $u\in I(G)^{\langle\ell\rangle}\subseteq I(G)^{\langle k+1\rangle}\subseteq\HS_1(I(G)^{k})$ by Theorem \ref{Thm:HS1-I(G)k}, as desired. Indeed, suppose for a contradiction that $e_1=\dots=e_\ell$. Then $u=e_1^\ell f$. Since $u\in\HS_1(I(G)^{k+1}):I(G)$, it follows that $ue_1=e_1^{\ell+1}f\in\HS_1(I(G)^{k+1})$. From Theorem~\ref{Thm:HS1-I(G)k}, we have $\HS_1(I(G)^{k+1})=\Lin_1(I(G)^{k+1})+I(G)^{\langle k+2\rangle}$, so either $ue_1\in\Lin_1(I(G)^{k+1})$ or $ue_1\in I(G)^{\langle k+2\rangle}$. By condition $(*)$ and the fact that $f\notin I(G)$, it follows that the only monomial $e\in\mathcal{G}(I(G))$ that divides $ue_1$ is $e=e_1$. Hence, the condition $(*)$ implies that $ue_1\notin\Lin_1(I(G)^{k+1})$. So, we should have $ue_1\in I(G)^{\langle k+2\rangle}$, but this is not possible because $e_1$ is the only monomial generator of $I(G)$ that divides $ue_1$. Hence $ue_1\notin\HS_1(I(G)^{k+1})$, which is a contradiction.\smallskip
		
		\textsc{Case 2.} Now, let $\ell=k$. Then $u=e_1\cdots e_kf$, with $f\notin I(G)$ and satisfying the condition $(*)$. It follows from Lemmas \ref{Lem:ee'} and \ref{Lem:ee'1} that $ue_1\notin I(G)^{\langle k+2\rangle}$ and $ue_1\notin\Lin_1(I(G)^{k+1})$. By Theorem \ref{Thm:HS1-I(G)k} we have $\HS_1(I(G)^{k+1})=\Lin_1(I(G)^{k+1})+I(G)^{\langle k+2\rangle}$. Hence $ue_1\notin\HS_1(I(G)^{k+1})$, which is absurd. Therefore, the condition $(*)$ can not hold and this concludes the proof.
	\end{proof}
	
	It would be of great interest to describe explicitly the sets $\Ass\,\HS_1(I(G)^k)$ for all $k\ge1$, as done in the case of $\Ass\,\HS_0(I(G)^k)=\Ass\, I(G)^k$ by Lam and Trung \cite{LT}.\smallskip
	
	As an immediate consequence of Theorem \ref{Thm:I(G)-Strong-Persistence} and Proposition \ref{Prop:ithSPP} we have
	\begin{Corollary}
		Let $G$ be a finite simple graph. Then, the functions $\height\,\HS_0(I(G)^k)$ and $\height\,\HS_1(I(G)^k)$ are decreasing.
	\end{Corollary}
	
	Combining Theorem \ref{Thm:I(G)-Strong-Persistence} with Corollary \ref{Cor:HSRI(G)} and Proposition \ref{prop:Min} we obtain immediately that
	\begin{Corollary}
		Let $G$ be a finite simple graph. Then, for all $k\ge1$,
		$$
		\Min\,\HS_1(I(G)^{k})\ =\ \Min\,\HS_1(I(G)).
		$$
	\end{Corollary}
	
	For the higher homological strong persistence properties, Theorem \ref{Thm:I(G)-Strong-Persistence} is not valid. Indeed, consider the edge ideal $I=(x_1x_2,x_1x_3,x_2x_3,x_1x_4,x_2x_5,x_3x_6)$. Using the \textit{Macaulay2} \cite{GDS} package \texttt{HomologicalShiftIdeals} \cite{FPack1} we checked that $(x_1),(x_2),(x_3)\in\Ass\,\HS_2(I)\setminus\Ass\,\HS_2(I^2)$. So, $I$ does not satisfy the second homological (strong) persistence property. Nonetheless, Conjecture \ref{Conj:HS-I(G)} holds for $I$.\smallskip
	
	It is an open question whether we have $\HS_i(I(G)^{k+1}):I(G)=\HS_i(I(G)^k)$ for all $k\gg0$ and all graphs $G$, when $i>1$.
	
	\section{Eventually homological linear powers of $I(G)$}\label{sec4}
	
	For the subsequent discussion we recall a few facts. Given monomials $u,v\in S$, we set $u:v=\lcm(u,v)/v$. A monomial ideal $I\subset S$ has \textit{linear quotients} if there exists an order $u_1,\dots,u_m$ of $\mathcal{G}(I)$ such that the ideals $(u_1,\dots,u_{j-1}):u_j$ are generated by variables for all $j=2,\dots,m$. For all $j$, we put $$\set(u_j)\ =\ \{i\ :\ x_i\in(u_1,\dots,u_{j-1}):u_j\}.$$ If $I\subset S$ is equigenerated and has linear quotients, then it has linear resolution. For a subset $F\subseteq[n]$, we set ${\bf x}_F=\prod_{i\in F}x_i$. In particular, we have ${\bf x}_{\emptyset}=1$. Let $I\subset S$ be a monomial ideal with linear quotients. By \cite[Lemma 1.5]{ET} we have
	$$
	\HS_i(I)\ =\ ({\bf x}_Fu\ :\ u\in\mathcal{G}(I),\ F\subseteq\set(u),\ |F|=i).
	$$
	
	We say that a monomial ideal $I\subset S$ has \textit{linear powers} if $I^k$ has linear resolution for all $k\ge1$. Whereas, we say that $I$ has \textit{homological linear powers} if $\HS_i(I^k)$ has linear resolution for all $i\ge0$ and all $k\ge1$. Finally, we say that $I$ has \textit{eventually homological linear powers} if $\HS_i(I^k)$ has linear resolution for all $i\ge0$ and all $k\gg0$.\smallskip
	
	Recall that a graph $G$ is called \textit{chordal} if it has no induced cycles of length bigger than three. By a result of Fr\"oberg \cite{Froberg88}, it is known that $I(G)$ has linear resolution if and only if $G^c$ is a chordal graph. Here, $G^c$ is the \textit{complementary graph} of $G$, that is the graph with $V(G^c)=V(G)$ and whose edges are the non-edges of $G$.\smallskip
	
	By a famous result of Herzog, Hibi and Zheng \cite{HHZ2004}, (see also \cite[Theorem 1.1]{F-HHZ} or \cite{BDMS} for a short proof), $I(G)$ has linear resolution, if and only if $I(G)$ has linear powers, if and only if, $I(G)^k$ has linear quotients for all $k\ge1$.\smallskip
	
	In view of many computational examples and partial evidence, we are tempted to pose the following conjecture.
	\begin{Conjecture}\label{Conj:HS-I(G)}
		Let $I(G)$ be an edge ideal with linear resolution. Then $I(G)$ has eventually homological linear powers.
	\end{Conjecture}
	
	In support of this conjecture, we have the next partial result.
	\begin{Theorem}\label{Thm:HS1-I(G)^k-Lin}
		Let $I(G)$ be an edge with linear resolution. Then, $\HS_1(I(G)^k)$ has linear quotients, and thus a linear resolution, for all $k\ge1$.
	\end{Theorem}
	\begin{proof}
		By \cite{HHZ2004} (see also \cite[Theorem 1.1]{F-HHZ} or \cite{BDMS}) $I(G)^k$ has linear quotients for all integers $k\ge1$. Then, the result follows from \cite[Theorem 1.3]{FH2023}.
	\end{proof}
	
	At the moment we do not have a general strategy to tackle Conjecture \ref{Conj:HS-I(G)}. Therefore, in what follows, we consider some special classes of edge ideals with linear resolution. In particular, we prove Conjecture \ref{Conj:HS-I(G)} for principal Borel edge ideals, initial lexsegment edge ideals, and edge ideals of complete multipartite graphs.\medskip
	
	\subsection{Principal Borel edge ideals} Let ${\bf 1}=(1,\dots,1)\in\mathbb{Z}_{\ge0}^n$, and let $u=x_ix_j\in S$ with $i<j$ be a squarefree monomial of degree two. The ideal
	$$
	B^{\bf 1}(u)\ =\ (x_rx_s\ :\ r<s,\, r\le i,\, s\le j)
	$$
	is called a \textit{principal Borel edge ideal}. By \cite[Theorem 3.9(b)]{FQ} we have
	\begin{Proposition}
		Principal Borel edge ideals have homological linear powers.
	\end{Proposition}
	
	\subsection{Initial lexsegment edge ideals} We denote by $>_{\lex}$, the lexicographical order on $S$ induced by $x_1> \cdots > x_n$. Let $v\in S$ be a squarefree monomial of degree $d$. The squarefree initial lexsegment ideal defined by $v$ is 
	\[
	\mathcal{L}^\mathrm{i}(v)\ =\ (w\in S:\ w \textit{ is a squarefree monomial with } \deg(w)=d,\ w \geq_{\lex} v).
	\]
	
	For a monomial $u=x_1^{a_1}\cdots x_n^{a_n}\in S$, we set $\deg_{x_i}(u)=a_i$ for all $i$.
	
	Let $v=x_ix_j$ be a squarefree monomial. We say that the edge ideal $\mathcal{L}^\mathrm{i}(v)$ is an \textit{initial lexsegment edge ideal}.  The powers of initial lexsegment edge ideals have been discussed in \cite[Proposition 3.2]{FMO} as recorded below.
	
	\begin{Proposition}
		Let $I=\mathcal{L}^\mathrm{i}(x_ix_j)$ be an initial lexsegment edge ideal and let $\mathcal{G}(I^k)=\{u_1 >_{\lex}\dots>_{\lex} u_m\}$. Then 
		\[
		(u_1, \ldots, u_{i-1}): (u_i) = (x_r : \deg_{x_r} (u_i) \leq k-1 \textit{ and } 1 \leq r \leq \max(u_i)-1)
		\]
	\end{Proposition}
	
	The above proposition gives us 
	\begin{equation}\label{eq:setu}
		\set(u)= \{r :\ \deg_{x_r} (u) \leq k-1 \textit{ and } 1 \leq r \leq \max(u)-1\}
	\end{equation}
	for any $u \in \mathcal{G}(I^k)$, where $I=\mathcal{L}^\mathrm{i}(x_ix_j)$.\medskip
	
	Recall that the \textit{support} of a monomial $u\in S$ is the set defined as $$\supp(u)\ =\ \{i:\ x_i\ \textit{divides}\ u\}.$$
	
	\begin{Theorem}\label{thm:initiallex}
		Initial lexsegment edge ideals have homological linear powers. 
	\end{Theorem}
	\begin{proof}
		Let $I=\mathcal{L}^\mathrm{i}(x_px_q)$ be an initial lexsegment edge ideal. We show that $\HS_j(I^k)$ has linear quotients with respect to the decreasing lexicographic ordering of the generators. Let ${\bf x}_F u$ and $ {\bf x}_G v$ be two minimal monomial generators of $\HS_j(I^k)$ with ${\bf x}_G v >_{\lex} {\bf x}_F u$, where $u,v \in \mathcal{G}(I^k)$, $F \subseteq \set(u)$, $G \subseteq \set(v)$ and $|F|=|G|=j$. We need to show that there exists some ${\bf x}_H w \in \mathcal{G}(\HS_j(I^k))$ with  ${\bf x}_H w >_{\lex} {\bf x}_F u$ such that $({\bf x}_G v) : ({\bf x}_F u )\subseteq ({\bf x}_H w ): ({\bf x}_F u )= (x_a)$ for some $a$. Let $s$ be the minimum integer such that $\deg_{x_i} ({\bf x}_Fu)=\deg_{x_i} ({\bf x}_Gv)$ for all $i <s$ and $\deg_{x_s} ({\bf x}_Fu)<\deg_{x_s} ({\bf x}_Gv)$. 
		
		Firstly, suppose $u=v$. Then, $s \in G\setminus F$. Let $a$ be the minimum element in $F\setminus G$. Since ${\bf x}_G u >_{\lex} {\bf x}_F u$, we have $s < a$. By setting $H= (F \setminus \{a\})\cup\{s\}$ and $w=u$, we obtain that ${\bf x}_Hw\in\mathcal{G}(\HS_j(I^k))$, ${\bf x}_H w >_{\lex} {\bf x}_F u$ and $({\bf x}_Hw) : ({\bf x}_Fu)=(x_{s})$.
		
		Now suppose that $u \neq v$. We first note that  $s <\max(u)$. To see this, assume that $s \geq \max(u)$. Using (\ref{eq:setu}) we observe that $i < \max(u)$ for any $i \in F$. Then due to our choice of $s$ and ${\bf x}_G v >_{\lex} {\bf x}_F u$, we have $\deg_{x_i} ({\bf x}_Fu)=\deg_{x_i} ({\bf x}_Gv)$ for all $i\in \supp({\bf x}_Fu)$. This implies that ${\bf x}_Fu$ divides ${\bf x}_Gv$, which contradicts the fact that ${\bf x}_G v $ is a minimal generator of $\HS_j(I^k)$. Hence, $s < \max(u)$.  
		
		Note that (\ref{eq:setu}) gives $\deg_{x_{s}}({\bf x}_Gv)\leq k$, and since  $\deg_{x_{s}} (u)  <\deg_{x_{s}} ({\bf x}_Gv) $, we obtain $\deg_{x_{s}} (u)  \leq k-1 $. Set $w=x_{s}(u/x_{\max(u)})$. Then $s \leq \max(w)$ because $s <\max(u)$. Following the proof of \cite[Proposition 3.2]{FMO}, we have $w \in I^k$. If $ s = \max(w)$, then $\deg_{x_t} (u) =1$ where $t=\max(u)$. By our choice of $s$ and since $i < \max(u)=t$ for any $i \in F$, it follows that ${\bf x}_Fu/x_t\in\mathcal{G}(\HS_j(I^k))$ divides ${\bf x}_Gv$, and $({\bf x}_Gv): ({\bf x}_Fu) =(x_s)$.
		
		Now suppose that $ s < \max(w)$. To show that ${\bf x}_F w \in \mathcal{G}(\HS_j(I^k))$, it only remains to prove that $F \subseteq \set(w)$. Observe that $\set(u)\setminus\{s\} \subseteq \set(w)$ because of (\ref{eq:setu}).  Therefore, $F \not\subseteq \set(w)$ if and only if $s \in F$ and $s\notin \set(w)$. Again, from (\ref{eq:setu}), we see that $s\notin \set(w)$ if and only if $k=\deg_{x_{s}} (w) $ or $ s = \max(w)$. Since $s<\max(w)$ by assumption, we may suppose that $k=\deg_{x_{s}} (w) $. Then $k=\deg_{x_{s}} (w) =\deg_{x_{s}} (u) +1$. Since $s \in F$, we obtain $k=\deg_{x_{s}} ({\bf x}_Fu) < \deg_{x_{s}} ({\bf x}_Gv)\leq k$, a contradiction. Hence $F \subseteq \set(w)$. Finally, ${\bf x}_Fw>_{\lex}{\bf x}_Fu$ and $({\bf x}_G v) : ({\bf x}_F u )\subseteq ({\bf x}_F w ): ({\bf x}_F u )= (x_s)$, as required.
	\end{proof}
	
	With a similar proof, one can show that final lexsegment edge ideals also have homological linear powers.
	
	\subsection{Edge ideals of complete multipartite graphs}
	Given two positive integers $a \leq b$, we set $[a,b]=\{c\in\mathbb{Z}: a\leq c\leq b\}$. Let $G$ be a graph with vertex set $V(G) = [n]$ and a vertex partition $V_1, \ldots, V_r$, where
	\[
	V_i=[t_{i-1}+1, t_i],\quad \text{ for all }\,\,\, i=1 , \ldots, r,
	\]
	with $t_0 = 0$ and $t_r = n$. If $E(G) = \{ \{i, j\} : i \in V_k, j \in V_\ell \text{ and } k \neq \ell\}$, then $G$ is called a \emph{complete multipartite} graph. It is easy to see that the complement graph $G^c$ of $G$ is chordal, and hence $I(G)$ has a linear resolution \cite{Froberg88}.
	
	Let $u\in S$ be a monomial. For any $1 \leq i \leq r$, we let $u_{V_i}$ denote the largest monomial $m$ dividing $u$ such that $\supp(m) \subseteq V_i$. Then, $\deg u_{V_i} = \sum_{j \in V_i} \deg_{x_j} u$. We define $\max(u)$ as the largest integer $i$ such that $i \in \supp(u)$.
	
	\begin{Lemma}\label{lem:powerpartite}
		With the notation introduced above, a monomial $u$ belongs to $\mathcal{G}(I^k)$ if and only if $\deg u = 2k$ and $\deg u_{V_i} \leq k$ for all $i = 1, \ldots, r$.
	\end{Lemma}
	\begin{proof}
		If $u \in \mathcal{G}(I^k)$, then $\deg u = 2k$, and it follows immediately from the definition of a complete multipartite graph that $\deg u_{V_i} \leq k$.
		
		To see the converse, let $u = x_{i_1} \cdots x_{i_{2k}}$ with $i_1 \leq \cdots \leq i_{2k}$. Given the assumption that $\deg u_{V_i} \leq k$ for all $i$, we observe that in the list $i_1, \ldots, i_{2k}$, each vertex in the first $k$ positions is adjacent to the corresponding vertex in the next $k$ positions. Specifically, for each $p = 1, \ldots, k$, the vertex in the $p$th position is adjacent to the vertex in the $(p+k)$th position because they belong to different partition sets. Hence, $u$ is a product of $k$ monomials in $I$, so $u \in \mathcal{G}(I^k)$.
	\end{proof}
	
	Next, we describe the linear quotients of the powers of edge ideals of complete multipartite graphs.
	\begin{Proposition}\label{prop:set}
		Let $G$ be a complete multipartite graph on $n$ vertices, and let $I = I(G)$. For all $k \geq 1$, $I^k$ has linear quotients with respect to the decreasing lexicographical order on $\mathcal{G}(I^k)$ induced by $x_1>\cdots>x_n$. Moreover, using the notation introduced above, for any $u \in \mathcal{G}(I^k)$, let $\max(u) \in V_d$ for some $d$. Then we have the following.
		\begin{enumerate}
			\item[\textup{(i)}] If there exists some $d'<d$ such that $\deg u_{V_{d'}} = k$, then 
			\[
			\set(u)=[1, {t_{d'-1}}] \cup [{t_{d'-1}+1},{\alpha-1} ] \cup [{t_{d'}+1}, {\max(u)-1}]
			\]
			where $\alpha$ is the maximal integer such that $\alpha\in\supp(u_{V_{d'}})$. In particular, $\ell < \max(u)$ for all $\ell \in \set(u)$. 
			\item[\textup{(ii)}] If $\deg u_{V_{d'}}< k$, for all $d'< d$ then $\set(u)=[1, \max(u)-1]$. In particular, $\ell < \max(u)$ for all $\ell \in \set(u)$. 
			\item[\textup{(iii)}] For any  $1 \leq i \leq n$ and for any $m \in \mathcal{G}(\HS_j(I^k))$, we have $\deg_{x_i}  m \leq k$. 
		\end{enumerate}
	\end{Proposition}
	
	\begin{proof}
		Let $u, v \in \mathcal{G}(I^k)$ with $u >_{\lex} v$, where $u = x_{i_1} \cdots x_{i_{2k}}$ and $v = x_{j_1} \cdots x_{j_{2k}}$. Then there exists an integer $p$ such that $i_\ell = j_\ell$ for $1 \leq \ell \leq p - 1$ and $i_p < j_p$. Let $i_p \in V_a$ and $j_p \in V_b$ for some $a$ and $b$. Since $i_p < j_p$, we have $a \leq b$. Define $w = x_{i_p} v / x_{j_p}$. Then $w >_{\lex} u$ and $(w) : (u) = (x_{i_p})$, where $x_{i_p}$ divides the generator of $(u) : (v)$. It remains to show that $w \in \mathcal{G}(I^k)$. If $a = b$, then for all $i$, $\deg w_{V_i} = \deg u_{V_i} \leq k$, so $w \in \mathcal{G}(I^k)$ by Lemma~\ref{lem:powerpartite}. If $a < b$, then by our choice of $p$, $\deg v_{V_b} < \deg u_{V_b} \leq k$, and thus $w \in \mathcal{G}(I^k)$ again by  Lemma~\ref{lem:powerpartite}.
		
		Statements (i) and (ii) follow immediately from the above argument together with Lemma~\ref{lem:powerpartite}. Indeed, let $u \in \mathcal{G}(I^k)$ with $\max(u) \in V_d$ for some $d$, and let $i \in \supp(u)$ and $j \in V(G)$. Set $w = x_j u / x_i$. Then $(w) :( u) = (x_j)$,  and $w >_{\lex} u$ if and only if $j < i$. By Lemma~\ref{lem:powerpartite},  $w \in \mathcal{G}(I^k)$ if and only if $\deg w_{V_t} \leq k$ for all $t = 1, \ldots, r$. This implies $w = x_j u / x_{\max(u)} \in \mathcal{G}(I^k)$ for all $j \in V_d$ with $j<\max(u)$ and for all $j \in V_{d'}$ with $\deg u_{V_{d'}} < k$ and $d' < d$. If $\deg u_{V_{d'}} = k$ for some $d'<d$ then $w = x_j u / x_\alpha \in \mathcal{G}(I^k)$ for all $j \in V_{d'}$ with $j<\alpha$ where $\alpha=\max\supp(u_{V_{d'}})$.
		
		The statement (iii) is a consequence of Lemma~\ref{lem:powerpartite} and (i) and (ii). 
	\end{proof}
	
	Now we are ready to prove our main result. 
	\begin{Theorem}
		Edge ideals of complete multipartite graphs have homological linear powers.
	\end{Theorem}
	\begin{proof}
		Let $G$ be a complete multipartite graph on $n$ vertices, with notation as introduced above, and let $I=I(G)$. Combining \cite[Theorem 2.3]{KNQ} with \cite[Theorem 2.2]{Bay019} we obtain that $\HS_j(I)$ is polymatroidal for all $j$, and thus $\HS_j(I)$ has liner resolution for all $j$, see \cite[Theorem 12.6.2]{HHBook}.
		
		Now let $k \geq 2$. We will show that $\HS_j(I^k)$ has linear quotients with respect to the decreasing lexicographic order on $\mathcal{G}(\HS_j(I^k))$ induced by $x_1 > \cdots > x_n$.
		
		Let $u'$ and $v'$ be two distinct elements in $\mathcal{G}(\HS_j(I^k))$ with $v' >_{\lex} u'$, and let $s$ be the smallest integer such that $\deg_{x_i} u' = \deg_{x_i} v'$ for all $i < s$ and $\deg_{x_s} u' < \deg_{x_s} v'$. We show that there exists a monomial $w'$ satisfying the following:
		\begin{equation}\label{w'}
			w' \in \mathcal{G}(\HS_j(I^k)) \text{ such that } (w'):(u')= (x_s) \text{ and } w'>_{\lex} u' .
		\end{equation} 
		Since $u', v' \in \mathcal{G}(\HS_j(I^k))$, we write $u' = {\bf x}_F u$ and $v' = {\bf x}_G v$ for some $u, v \in \mathcal{G}(I^k)$ and $F \subseteq \set(u)$, $G \subseteq \set(v)$ with $|F| = |G| = j$.
		
		If $u = v$, then $s \in G \setminus F$ and there exists some $s' \in F \setminus G$ such that $s' > s$. Let $F' = (F \setminus \{s'\}) \cup \{s\}$. Then $w' = {\bf x}_{F'} u \in \mathcal{G}(\HS_j(I^k))$ satisfies (\ref{w'}).
		
		Now assume that $u \neq v$. Let $r$ be the smallest integer for which $r\in\supp(u)$ such that $s < r$. Note that such an $r$ exists: if $\max(u) \leq s$, then $\ell < s$ for all $\ell \in F$ since $\ell < \max(u)$ by Proposition~\ref{prop:set}(i) and (ii). Given $\deg_{x_i} u' = \deg_{x_i} v'$ for all $i < s$, we conclude that $u'$ divides $v'$, contradicting $v' \in \mathcal{G}(\HS_j(I^k))$.
		
		Let $s \in V_a$ for some $a$. Recall that $\deg u_{V_a} \leq k$ due to Lemma~\ref{lem:powerpartite}. We distinguish the following cases.\medskip
		
		\textsc{Case 1.} Assume $\deg u_{V_a} = k$ and $r \notin V_a$. Since $s < r$ and $s \in V_a$, the partition construction gives $r \in V_\ell$ for some $\ell > a$. By the choice of $r$, we have $s \geq \max \{i : i \in \supp(u_{V_a})\}$, so $s \notin \set(u)$ by Proposition~\ref{prop:set}(i) and hence $s \notin F$. Also, Proposition~\ref{prop:set}(iii) implies $\deg_{x_s} v' \leq k$, so $\deg_{x_s} u \leq \deg_{x_s} u' < k$. Since $\deg u_{V_a} = k$, there exists some $b \in V_a \setminus \{s\}$ with $b \in \supp(u_{V_a})$ and $b < s$.\smallskip\\
		\textbf{Claim:} There exists some $b \in \supp(u_{V_a})$ with $b < s$ such that $b \notin F$.\smallskip\\
		\textbf{Proof of Claim:} Assume, for contradiction, that no such $b$ exists, meaning $b \in F$ for every $b \in \supp(u_{V_a})$ with $b < s$. Then $\deg_{x_b} u' = \deg_{x_b} u + 1$. Given our choice of $s$, for all such $b$ we get $\deg_{x_b} v' = \deg_{x_b} u' = \deg_{x_b} u + 1$, and since $\deg u_{V_a} = k$, 
		we obtain $u_{V_a} =v_{V_a}$. Therefore, $s \geq \max \{i : i \in \supp(v_{V_a})\}$, implying $s \notin G$ by Proposition~\ref{prop:set}(i). This leads to $\deg_{x_s} u = \deg_{x_s} v$, contradicting $\deg_{x_s} u < \deg_{x_s} v$. Thus, the claim holds.
		
		We fix $b$ as the element satisfying our claim.  Since $s \notin F$ (because $s \notin \set(u)$), if there exists $a \in F$ with $a > s$, set $F' = (F \setminus \{a\}) \cup \{b\}$ and $w = x_s u / x_b$. By Lemma~\ref{lem:powerpartite}, $w \in \mathcal{G}(I^k)$, and since $b<s$, applying Proposition~\ref{prop:set}, we obtain $F'\subseteq \set(w')$. Thus, $w'={\bf x}_{F'}w \in \mathcal{G}(\HS_j(I^k))$ satisfies (\ref{w'}). 
		
		On the other hand, if $a < s$ for all $a \in F$, then there exists $c \in F \cap V_p$ such that $p < a$.  Indeed, if no such $c$ exists, then $F\subseteq V_a$ and $\deg u'_{V_a}= k+j=\deg v'_{V_a}= k+j$ because $\deg_{x_k} u' = \deg_{x_k} v'$ for all $k < s$. Then ${\bf x}_Fu_{V_a}={\bf x}_Gv_{V_a}$, contradicting $\deg_{x_s} u < \deg_{x_s} v$. Therefore, we can choose $c \in F\cap V_p$ with $p< a$. Set $w=(x_cx_s u )/x_rx_b$, where $b$ is the integer fixed above, and  $F'= \{b\} \cup (F\setminus \{c\})$.  Observe that $\deg u_{V_p}<k$ because $\deg u_{V_a}=k$, $1<\deg u_{V_\ell}$ where $r \in V_\ell$, and $\deg u=2k$. Hence $\deg w_{V_p}\leq k$ and $w \in I^k$ due to Lemma~\ref{lem:powerpartite} and $F'\subseteq \set (w)$ due to Proposition~\ref{prop:set}(i). Therefore $w'={\bf x}_{F'}w \in \mathcal{G}(\HS_j(I^k))$ satisfies (\ref{w'}). \medskip
		
		\textsc{Case 2.} Assume $\deg u_{V_a}= k$, and  $r\in V_a$. If $u'$ satisfies one of the following conditions:  $r\neq \max(u_{V_a}) $; $r = \max(u_{V_a}) $ and $\deg_{x_r} u>1$; $r = \max(u_{V_a}) $, $ \deg_{x_r} u=1$, and there does not exist any $d \in F$ such that $s \leq  d < r$, then 
		let $w=x_su/x_r$. By Lemma~\ref{lem:powerpartite}, $w \in \mathcal{G}(I^k)$, and using Proposition~\ref{prop:set}(i), we have $F\subseteq \set(w)$. Therefore, $w'={\bf x}_{F}w \in \mathcal{G}(\HS_j(I^k))$ and $w'$ satisfies (\ref{w'}). 
		
		One the other hand, if none of the above conditions hold, then the only remaining case is when $r = \max(u_{V_a})$, $\deg_{x_r} u=1$, and there exists some $d \in F$ such that $s \leq  d < r$.  Let $c$ be the maximum integer in $F$ such that $s \leq c < r$. If $s\notin F$, then set $F'=(F\setminus \{c\} ) \cup \{s\}$ and $w=x_cu/x_r$. By Lemma~\ref{lem:powerpartite}, $w \in \mathcal{G}(\HS_j(I^k))$, and by  Proposition~\ref{prop:set}, $F'\subseteq \set(w)$. Thus, $w'={\bf x}_{F'}w \in \mathcal{G}(\HS_j(I^k))$ and $w'$ satisfies (\ref{w'}). 
		
		Now suppose $ s \in F$. Then  $\deg_{x_s} u'=\deg_{x_s} u+1$. We first prove the following\smallskip\\
		{\bf Claim:} There exists some $b \in \supp(u_{V_a})$ with $b < s$ such that $b \notin F$.\smallskip\\
		{\bf Proof of claim:} For $s <  d < r$, we have $d \notin \supp(u)$ because of the choice of $r$. Since $r = \max(u_{V_a}) $ with $\deg_{x_r} u=1$ and $\deg u_{V_a}= k\geq 2$, there exists some $b \in \supp(u_{V_a})$ with $b \leq s<r$. If $b \in F$ for each $b \in \supp(u_{V_a})$ with $b\leq s$, then $\deg_{x_b} u'\geq 2$. Then $\deg_{x_b} v'\geq 2$. Since ${\bf x}_G$ is squarefree, we obtain $u_{V_a}/x_r$ divides $v_{V_a}$ and $\deg v_{V_a}\geq k-1$. If $s\notin G$, then $\deg_{x_s} v'=\deg_{x_s} v$, which eventually contradicts $\deg_{x_s} u' < \deg_{x_s}v'$ because $\deg_{x_s} v \leq \deg_{x_s} u+1$ (since $\deg u_{V_a}/x_r=k-1$), and $\deg_{x_s} u'=\deg_{x_s} u+1$.  This shows that $s\in G$. By Proposition~\ref{prop:set},  $s \in G$ is possible only if either $ \deg v_{V_a}=k-1$, or $\deg v_{V_a}=k $ and  $s< \max(v_{V_a})$. In both cases, we would have $\deg_{x_s} v' \leq \deg_{x_s} u+1=\deg_{x_s} u'$, which again contradicts $\deg_{x_s} u' < \deg_{x_s}v'$. Therefore, our assumption is false and there exists some $b \leq s$ with $b \in \supp(u_{V_a})$ and $b \notin F$. Since $ s \in F$, we have $b<s$ and thus our claim holds. 
		
		Let $b \in \supp(u_{V_a})$ with $b \neq r$ and $b \notin F$. Set $F'=(F\setminus \{c\} ) \cup \{b\}$ and $w=x_sx_cu/x_rx_b$. By Lemma~\ref{lem:powerpartite}, $w \in \mathcal{G}(I^k)$, and by Proposition~\ref{prop:set}, $F'\subseteq \set(w)$. Therefore, $w'={\bf x}_{F}w \in \mathcal{G}(\HS_j(I^k))$ and $w'$ satisfies (\ref{w'}). \medskip
		
		\textsc{Case 3.} Assume $\deg u_{V_a} \leq k-1$ and $r\notin V_a$. If $\deg u_{V_a}< k-1$, or $\deg u_{V_a} = k-1$ and no $d \in F$ satisfies $s \leq  d \leq \max(V_a)$, then set $w=x_su/x_r$. Since $\deg w_{V_a} =\deg u_{V_a}+1 \leq k$ and $\max(u_{V_a}) \leq \max(w_{V_a})=s$, Lemma~\ref{lem:powerpartite} gives $w=x_su/x_r \in \mathcal{G}(I^k)$ and Proposition~\ref{prop:set}(i) and (ii) gives $F \subseteq \set(w)$. Therefore, $w'={\bf x}_Fw \in  \mathcal{G}(\HS_j(I^k))$ satisfies (\ref{w'}).
		
		Now assume $\deg u_{V_a} = k-1$ and there exists some $d \in F$ such that $s \leq  d \leq \max(V_a)$. Let $c$ be the maximum in $F$ such that $s \leq c \leq \max(V_a)$. If $s\notin F$, set $F'=(F\setminus \{c\} ) \cup \{s\}$ and $w=x_cu/x_r$. By Lemma~\ref{lem:powerpartite}, $w \in \mathcal{G}(\HS_j(I^k))$ and by Proposition~\ref{prop:set}(i), $F'\subseteq \set(w)$. Thus, $w'={\bf x}_{F'}w \in \mathcal{G}(\HS_j(I^k))$ satisfies (\ref{w'}). Now, suppose that $ s \in F$. Then  $\deg_{x_s} u'=\deg_{x_s} u+1$. We first prove the following\smallskip\\
		{\bf Claim:} There exists some $b \in \supp(u_{V_a})$ with $b < s$ such that $b \notin F$.\smallskip\\
		{\bf Proof of claim:} Since $k \geq 2$, $\supp(u_{V_a})\neq \emptyset$. Then by choice of $r$, for all $b \in \supp(u_{V_a})$ we have $b \leq s$. Assume $b \in F$ for all $b \in \supp(u_{V_a})$. Then $\deg_{x_b} v'=\deg_{x_b} u'\geq 2$. Since ${\bf x}_G$ is squarefree, we obtain $k-1 \leq \deg v_{V_a}$ and $u_{V_a}$ divides $v_{V_a}$.
		\begin{enumerate}
			\item[$\bullet$] If $\deg v_{V_a} = k-1$, then $u_{V_a}=v_{V_a}$, and $\deg_{x_s} v'$  is at most $\deg_{x_s}u'$. 
			\item[$\bullet$]  If $\deg v_{V_a} = k$ and $s \notin G$, then $\deg_{x_s} v'= \deg_{x_s} v$, and $\deg_{x_s} v'\le\deg_{x_s}u'$.
			\item[$\bullet$]   If $\deg v_{V_a} = k$ and  $s\in G$, then  by Proposition~\ref{prop:set}, we have $s< \max(v_{V_a})$. This gives  $\deg_{x_s} v =\deg_{x_s} u$, and $\deg_{x_s} v'= \deg_{x_s} v+1= \deg_{x_s}u'$. 
		\end{enumerate}  
		In all three cases above, we have $\deg_{x_s} v' \leq \deg_{x_s} u'$, contradicting $\deg_{x_s} u' < \deg_{x_s}v'$. This shows that  there exists some $b \leq s$ with $b \in \supp(u_{V_a})$ and $b \notin F$. Since $ s \in F$, we obtain $b<s$ and our claim holds. 
		
		Let $b \in \supp(u_{V_a})$ with $b < s$ such that $b \notin F$. Set $F'=(F\setminus \{c\} ) \cup \{b\}$ and $w=x_sx_cu/x_rx_b$. By Lemma~\ref{lem:powerpartite}, $w \in \mathcal{G}(I^k)$ and Proposition~\ref{prop:set}(i) gives $F'\subseteq \set(w)$.  Thus, $w'={\bf x}_{F}w \in \mathcal{G}(\HS_j(I^k))$ satisfies (\ref{w'}). \medskip
		
		\textsc{Case 4.} Assume $\deg u_{V_a} \leq k-1$ and $x_r\in V_a$. This case is similar to Case 2 to some extent; however, here we compare $r$ with $\max(u)$ instead of $\max(u_{V_a})$. 
		
		If $r \neq \max(u)$ or $\deg_{x_\alpha} u >1$, or $r = \max(u)$ with $\deg_{x_r} u =1$ and no $d \in F$ satisfies $s \leq d < r$, then set $w=x_su/x_r$. By Lemma~\ref{lem:powerpartite}, $w=x_su/x_r \in \mathcal{G}(I^k)$, and Proposition~\ref{prop:set}(i) and(ii) imply $F \subseteq \set(w)$. Therefore   $w'={\bf x}_Fw \in  \mathcal{G}(\HS_j(I^k))$ satisfies (\ref{w'}).
		
		The only case remain to be discussed is when $r = \max(u)$ with $\deg_{x_r} u =1$ and there exists some $d \in F$ satisfies $s \leq d < r$. Let $c$ be the maximum in $F$ such that $s \leq  c \leq r$. If $s\notin F$, set $F'=(F\setminus \{c\} ) \cup \{s\}$ and $w=x_cu/x_r$. By Lemma~\ref{lem:powerpartite}, $w \in \mathcal{G}(\HS_j(I^k))$, and by Proposition~\ref{prop:set}(i) gives $F'\subseteq \set(w)$. Then $w'={\bf x}_{F'}w \in \mathcal{G}(\HS_j(I^k))$ satisfies (\ref{w'}). Now suppose that $ s \in F$. Then  $\deg_{x_s} u'=\deg_{x_s} u+1$. We first prove the following\smallskip\\
		{\bf Claim:} There exists some $b \in \supp(u)$ with $b < s$ such that $b \notin F$.\smallskip\\
		{\bf Proof of claim:} Since $r= \max(u)$, there exists some $b \in \supp(u)$ such that $b \in V_t$ for some $t < a$, also $b<s$ because of $s \in V_a$ and the choice of $r$.  If $b \in F$, for all $b \in \supp(u)$ with $b < s$, then $\deg_{x_b} v'=\deg_{x_b} u'\geq 2$. Since ${\bf x}_G$ is squarefree, we obtain $u/x_r$ divides $v$. 
		\begin{enumerate}
			\item[$\bullet$]  If $s \notin G$, then $\deg_{x_s} v'= \deg_{x_s} v$ and $\deg_{x_s} v$ is at most $\deg_{x_s} u+1=\deg_{x_s}u'$.
			\item[$\bullet$]   If $s\in G$, then Proposition~\ref{prop:set} gives $s< \max(v_{V_a})$. This implies $\deg_{x_s} v =\deg_{x_s} u$, and $\deg_{x_s} v'= \deg_{x_s} v+1= \deg_{x_s}u'$. 
		\end{enumerate}  
		In both of the above cases, we have $\deg_{x_s} v' \leq \deg_{x_s} u'$, contradicting $\deg_{x_s} u' < \deg_{x_s}v'$. Thus our claim holds.
		
		As in Case 3, let $b \in \supp(u_{V_a})$ with $b < s$ such that $b \notin F$. Set $F'=(F\setminus \{c\} ) \cup \{b\}$ and $w=x_sx_cu/x_rx_b$. By Lemma~\ref{lem:powerpartite}, $w \in \mathcal{G}( I^k)$ and by Proposition~\ref{prop:set}(i) we have $F'\subseteq \set(w)$. Therefore $w'={\bf x}_{F}w \in \mathcal{G}(\HS_j(I^k))$ satisfies (\ref{w'}).
	\end{proof}
	
	We conclude the paper, with the following question.
	\begin{Question}
		Let $I(G)$ be an edge ideal with linear resolution. Is it true that
		$$
		\v(\HS_i(I(G)^k))=2k+(i-1)
		$$
		for all $k\gg0$ and all $i\le\pd\,I(G)^k$?
	\end{Question}
	
	For $i=0$ the above equation holds by \cite[Theorem 5.1]{F2023}.\bigskip
	
	\noindent\textbf{Acknowledgment.}
	A. Ficarra was partly supported by INDAM (Istituto Nazionale di Alta Matematica), and also by the Grant JDC2023-051705-I funded by
	MICIU/AEI/10.13039/501100011033 and by the FSE+. A.A. Qureshi is supported by Scientific and Technological Research Council of Turkey T\"UB\.{I}TAK under the Grant No: 124F113, and is thankful to T\"UB\.{I}TAK for their support.

\vspace*{-0.4cm}
	
	\begin{thebibliography}{99}
		\bibitem{Ban24} S. Bandari, \textit{Polymatroidal ideals and linear resolution}, Journal of Algebraic Systems, {\bf 11}(2024) 147-153.
		
		\bibitem{AB} A. Banerjee, \textit{The regularity of powers of edge ideals}, J. Algebraic Combin. {\bf41}(2015), no. 2, pp. 303--321.
		
		\bibitem{BDMS} E. Basser, R. Diethorn, R. Miranda, M. Stinson-Maas, \textit{Powers of Edge Ideals with Linear Quotients}, 2024, preprint \url{https://arxiv.org/abs/2412.03468}
		
		\bibitem{Bay019} S. Bayati, \textit{Multigraded shifts of matroidal ideals}, Arch. Math., (Basel) {\bf 111} (2018), no. 3, 239--246.
		
		\bibitem{Bay2023} S. Bayati, \textit{A Quasi-additive Property of Homological Shift Ideals}. Bulletin of the Malaysian Mathematical Sciences Society, 2023, 46(3), p.111.
		
		\bibitem{BJT019} S. Bayati, I. Jahani, N. Taghipour, \textit{Linear quotients and multigraded shifts of Borel ideals}, Bull. Aust. Math. Soc. 100 (2019), no. 1, 48--57.
		
		\bibitem{BMS24} P. Biswas, M. Mandal, K. Saha, \textit{Asymptotic behaviour and stability index of v-numbers of graded ideals}, 2024, preprint \url{arXiv:2402.16583}.
		
		\bibitem{B79} M. Brodmann, \textit{Asymptotic stability of $\textup{Ass}(M/I^nM)$}, Proc. Am. Math. Soc., 74(1979), 16--18
		
		\bibitem{B79a} M. Brodmann, \textit{The asymptotic nature of the analytic spread}, Math. Proc. Cambridge Philos. Soc., {\bf 86} (1979), 35--39.
		
		\bibitem{Conca23} A. Conca, \textit{A note on the v-invariant}, Proceedings of the American Mathematical Society, 2024, 152(6), pp. 2349--2351
		
		\bibitem{CF1} M. Crupi, A. Ficarra, \textit{Very well--covered graphs by Betti splittings}, J. Algebra {\bf 629}(2023) 76--108. https://doi.org/10.1016/j.jalgebra.2023.03.033.
		
		\bibitem{CF2} M. Crupi, A. Ficarra, \textit{Very well-covered graphs via the Rees algebra}, Mediterranean Journal of Mathematics 21.4 (2024): 135.
		
		\bibitem{CHT99} S.D. Cutkosky, J. Herzog, N.V. Trung. {\it Asymptotic Behaviour of the Castelnuovo--Mumford regularity}. Compositio Mathematica {\bf 118}, (1999), 243--261.
		
		\bibitem{Dirac61} G.A. Dirac, \textit{On rigid circuit graphs}, Abh. Math. Sem. Univ. Hamburg, {\bf 38} (1961), 71--76.
		
		\bibitem{FMO} C. Ferr\'o, M. Murgia, O. S. Olteanu \textit{Powers of edge ideals}, Le Matematiche, {\bf 67}(1), (2012), 129-144.
		
		\bibitem{F2} A. Ficarra, \textit{Homological shifts of polymatroidal ideals}, 2024, to appear in  Bull. math. Soc. Sci. Math. Roum., available at  \url{https://arxiv.org/abs/2205.04163}
		
		\bibitem{F2023}  A. Ficarra, \textit{Simon conjecture and the $\v$-number of monomial ideals}, Collectanea Mathematica (2024): 1-16. https://doi.org/10.1007/s13348-024-00441-z
		
		\bibitem{F-HHZ} A. Ficarra, {\em A new proof of the Herzog-Hibi-Zheng theorem}, 2024, preprint \url{https://arxiv.org/abs/2409.15853}.
		
		\bibitem{FPack1} A. Ficarra, \textit{Homological Shift Ideals: Macaulay2 Package}, 2023, preprint \url{https://arxiv.org/abs/2309.09271}.
		
		\bibitem{F-SCDP} A. Ficarra, \textit{Shellability of componentwise discrete polymatroids}, Electron. J. Comb., Volume 32, Issue 1 (2025), P1.41.
		
		\bibitem{FH2023} A. Ficarra, J. Herzog, \textit{Dirac’s Theorem and Multigraded Syzygies}. Mediterr. J. Math. 20, 134 (2023). https://doi.org/10.1007/s00009-023-02348-8
		
		\bibitem{FQ} A. Ficarra, A.A. Qureshi, \textit{The homological shift algebra of a monomial ideal}, (2024), preprint \url{https://arxiv.org/abs/2412.21031}.
		
		\bibitem{FS2} A. Ficarra, E. Sgroi, \textit{Asymptotic behaviour of the $\v$-number of homogeneous ideals}, 2023, preprint \url{https://arxiv.org/abs/2306.14243}.
		
		\bibitem{FSPack} A. Ficarra, E. Sgroi, \texttt{VNumber}, \textit{Macaulay2 Package} available at \url{https://github.com/EmanueleSgroi/VNumber}, 2024.
		
		\bibitem{FSPackA} A. Ficarra, E. Sgroi, \textit{Asymptotic behaviour of integer programming and the $\v$-function of a graded filtration}, 2024, preprint \url{https://arxiv.org/abs/2403.08435}.
		
		\bibitem{Froberg88} R. Fr\"oberg, \textit{On Stanley-Reisner rings}, Topics in algebra, Part 2 (Warsaw, 1988), 57--70, Banach Center Publ., 26, Part 2, PWN, Warsaw, 1990.
		
		\bibitem{GDS} D.~R.~Grayson, M.~E.~Stillman. {\em Macaulay2, a software system for research in algebraic geometry}. Available at \url{http://www.math.uiuc.edu/Macaulay2}.
		
		\bibitem{JS} A. V. Jayanthan, S. Selvaraja, \textit{Upper bounds for the regularity of powers of edge ideals of graphs}, Journal of Algebra 574, (2021), 184–205.
		
		\bibitem{KNQ} K. Khashyarmanesh, M. Nasernejad, and A. A. Qureshi, \textit{On the matroidal path ideals}, J. Algebra Appl. (2023), 2350227.
		
		\bibitem{HNTT} H.T. H\`a, H. Nguyen, N. Trung, T. Trung, \textit{Depth functions of powers of homogeneous ideals}, Proc. AMS, 149 (2021), 1837–1844.
		
		\bibitem{HHBook} J.~Herzog, T.~Hibi, \emph{Monomial ideals}, Graduate texts in Mathematics {\bf 260}, Springer, 2011.
		
		\bibitem{HHZ2004} J. Herzog, T. Hibi and X. Zheng, \textit{Monomial ideals whose powers have a linear resolution},	Math. Scand. (2004), 23--32.
		
		\bibitem{HMRZ021a} J. Herzog, S. Moradi, M. Rahimbeigi, G. Zhu, \textit{Homological shift ideals}. Collect. Math. {\bf 72} (2021), 157--74.
		
		\bibitem{HMRZ021b} J. Herzog, S. Moradi, M. Rahimbeigi, G. Zhu, \textit{Some homological properties of borel type ideals}, Comm. Algebra {\bf 51} (4) (2023) 1517--1531.
		
		\bibitem{HQ} J. Herzog, A.A. Qureshi, \textit{Persistence and stability properties of powers of ideals}, J. Pure Appl. Algebra, 219 (2015), 530--542.
		
		\bibitem{ET} J.~ Herzog, Y. Takayama, \textit{Resolutions by mapping cones}, in: The Roos Festschrift volume Nr.2(2), Homology, Homotopy and Applications {\bf 4}, (2002), 277--294.
		
		\bibitem{Kod99} V. Kodiyalam. {\it Asymptotic Behaviour of Castelnuovo-Mumford regularity}. Proc. Amer. Math. Soc. {\bf 198}(2), (1999), 407--411.
		
		\bibitem{LT} H.M. Lam, N.V. Trung, \textit{Associated primes of powers of edge ideals and ear decompositions of graphs}, Trans. Amer. Math. Soc. 372 (2019), 3211–3236.
		
		\bibitem{LW} D. Lu, Z. Wang, \textit{The resolutions of generalized co-letterplace ideals and their powers}, Journal of Algebra {\bf673}(2025): 321-350.
		
		\bibitem{MMV} J. Mart\'inez-Bernal, S. Morey, R. Villarreal, \textit{Associated primes of powers of edge ideals}, Collect. Math. {\bf63}, 361-374 (2012).
		
		\bibitem{MPNauty} B. D. McKay, A. Piperno, Practical Graph Isomorphism, II, J. Symbolic Comput. 60 (2014), 94–112. The software package Nauty is available at \url{http://cs.anu.edu.au/~bdm/nauty/}
		
		\bibitem{NP} E. Nevo, I. Peeva, \textit{$C_4$-free edge ideals}, J. Algebr. Combin. 37, 243–248 (2013).
		
		\bibitem{Rat} L. J. Ratliff Jr., \textit{On prime divisors of $I^n$, $n$ large}, Mich. Math. J. 23 (1976), 337--352.
		
		\bibitem{TBR24} N. Taghipour, S. Bayati, F. Rahmati, \textit{Homological linear quotients and edge ideals of graphs}. Bulletin of the Australian Mathematical Society (2024): 1-12.
	\end{thebibliography}
\end{document}